\documentclass[12pt,a4paper]{article}
\usepackage[utf8]{inputenc}
\usepackage[T1]{fontenc}
\usepackage{amsmath}
\usepackage{amsthm}
\usepackage{amsfonts}
\usepackage{amssymb}
\usepackage{graphicx}
\usepackage{authblk}
\usepackage{hyperref}
\usepackage{tikz}
\usepackage{wrapfig}
\usepackage{caption}
\usepackage{subcaption}
\usepackage[capposition=top]{floatrow}
\usepackage[document]{ragged2e}
\usepackage[left=2.00cm, right=2.00cm, top=2.00cm, bottom=2.00cm]{geometry}

\newtheorem{theorem}{Theorem}
\newtheorem{lemma}{Lemma}
\newtheorem{property}{Property}
\newtheorem{corollary}{Corollary}

\title{\textbf{CHARACTERIZING ALTERNATING SIGN TRIANGLES}\thanks{This project was supported by the University of Minnesota's Office of Undergraduate Research.}}
\author{SON NGUYEN}
\affil{University of Minnesota - Twin Cities\\\href{mailto:nguy4309@umn.edu}{nguy4309@umn.edu}}
\date{}

\begin{document}
\setlength{\abovedisplayskip}{0pt}
\setlength{\belowdisplayskip}{0pt}
\setlength{\abovedisplayshortskip}{0pt}
\setlength{\belowdisplayshortskip}{0pt}
	
	\maketitle
	
	\begin{abstract}
	\justify
		Alternating sign triangles were introduced by Carroll and Speyer in relation to cube recurrence, by analogy to alternating sign matrices for octahedron recurrence. Permutation triangles are the alternating sign triangles whose entries are either $0$ or $1$, by analogy with permutation matrices. In this paper, we prove a simple characterization of permutation triangles, originally conjectured by Glick. We will also prove some properties of alternating sign triangles.

	\end{abstract}
	
	\section{Introduction}
	
	\justify
	The \textit{octahedron recurrence} is given by the initial conditions $f_{i,j,k}=x_{i,j,k}$ for $k = -1,0$ and
	
	\begin{align*}
	    f_{i,j,k-1}f_{i,j,k+1} &= f_{i-1,j,k}f_{i+1,j,k}+ \lambda f_{i,j-1,k}f_{i,j+1,k}. & (k\geq0)
	\end{align*}
	
	\justify
	An \textit{alternating sign matrix} is a square matrix that satisfies:
	
	\begin{enumerate}
	    \item all entries are $-1,0,\text{ or } 1$,
	    \item every row and column has sum $1$,
	    \item in every row and column the non-zero entries alternate in sign.
	\end{enumerate}
	
	\justify
	In \cite{robbins}, Robbins and Rumsey found that the exponents of the $x_{i,j,k}$ in any monomial in $f_{i_0,j_0,k_0}$ formed an alternating sign matrix.
	
	In \cite{propp}, James Propp introduced another recurrence called the \textit{cube recurrence}, which is given by $f_{i,j,k} = x_{i,j,k}$ for $i+j+k=-1,0,1$ and
	
	\begin{align*}
	    f_{i,j,k}f_{i-1,j-1,k-1}&=f_{i-1,j,k}f_{i,j-1,k-1}+f_{i,j-1,k}f_{i-1,j,k-1}+f_{i,j,k-1}f_{i-1,j,k-1}. & (i+j+k>1)
	\end{align*}
	
	\justify
	In \cite{carroll}, Carroll and Speyer introduced a combinatorial object that describe the cube recurrence called \textit{grove}, which will be defined rigorously in section 2. Carroll and Speyer proved that $f_{0,0,0}$ is a sum of Laurent monomials in the variables $x_{i,j,k}$, and in each monomial, the exponent of $x_{i,j,k}$ is either $-1,0,\text{ or } 1$. Carroll and Speyer also observed that the exponents of $x_{i,j,k}$ form a sort of alternating sign triangles. In section 3, we will give our definition for alternating sign triangles. In section 4, we will prove a simple characterization of \textit{permutation triangles}, which is a special case of alternating sign triangles in which there is no $-1$ entry. In section 5, we will prove some properties of alternating sign triangles.
    
    \textbf{Acknowledgments.} I thank Max Glick for his conjectures about permutaion triangles, and for his valuable comments on the first version of this paper, and I thank Pavlo Pylyavskyy for his helpful suggestions and continuous support.
	
	\section{Background}
	
	\justify
	In \cite{carroll}, Carroll and Speyer defined groves as follows.
	
	\justify
	Define the \textit{lower cone} of any $(i,j,k)\in \mathbb{Z}^{3}$ to be
	
	\begin{align*}
	    C(i,j,k) &= \left\{ (i',j',k')\in \mathbb{Z}^{3} | i'\leq i,j' \leq j,k'\leq k \right\}.
	\end{align*}
	
	\justify
    Let $\mathcal{L}\subseteq \mathbb{Z}^{3}$ be a subset such that, whenever $(i, j, k) \in \mathcal{L}, C(i, j, k) \subseteq \mathcal{L}$. Let $\mathcal{U} = \mathbb{Z}^{3}-\mathcal{L}$, and define the set of \textit{initial conditions}
    
    \begin{align*}
        \mathcal{I} = \left\{(i, j, k) \in \mathcal{L} | (i + 1, j + 1, k + 1) \in \mathcal{U}\right\}.
    \end{align*}
    
    \justify
    We also define a \textit{rhombus} to be any set of the form
    
    \begin{align*}
        r_{a}(i, j, k) &= \left\{(i, j, k),(i, j - 1, k),(i, j, k - 1),(i, j - 1, k - 1)\right\}\\
        r_{b}(i, j, k) &= \left\{(i, j, k),(i - 1, j, k),(i, j, k - 1),(i - 1, j, k - 1)\right\}\\
        r_{c}(i, j, k) &= \left\{(i, j, k),(i - 1, j, k),(i, j - 1, k),(i - 1, j - 1, k)\right\}
    \end{align*}
    
    \justify
    In addition, define the \textit{edges} of each rhombus to be the pairs
    
    \begin{align*}
        e_{a}(i, j, k) &= \left\{(i, j - 1, k),(i, j, k - 1)\right\} & e'_{a}(i, j, k) &= \left\{(i, j, k),(i, j - 1, k - 1)\right\}\\
        e_{b}(i, j, k) &= \left\{(i - 1, j, k),(i, j, k - 1)\right\} & e'_{b}(i, j, k) &= \left\{(i, j, k),(i - 1, j, k - 1)\right\}\\
        e_{c}(i, j, k) &= \left\{(i - 1, j, k),(i, j - 1, k)\right\} & e'_{c}(i, j, k) &= \left\{(i, j, k),(i - 1, j - 1, k)\right\}
    \end{align*}
    
    \justify
    Now suppose that $N$ is a cutoff for $\mathcal{I}$. Define $\mathcal{G}$ be a graph whose vertices are the points in $\mathcal{I}$ and edges are the edges of all rhombi occurring in $\mathcal{I}$. We define an $\mathcal{I}$-grove within radius $N$ to be a subgraph $G \subseteq \mathcal{G}$ with the following properties:
    
    \begin{itemize}
        \item (Completeness) the vertex set of $G$ is all of $\mathcal{I}$;
        \item (Complementarity) for every rhombus, exactly one of its two edges occurs in $G$;
        \item (Compactness) for every rhombus all of whose vertices satisfy $i+j+k<-N$, the short edge occurs in $G$;
        \item (Connectivity) every component of $G$ contains exactly one of the following sets of vertices, and conversely, each such set is contained in some component:
        \begin{itemize}
            \item $\left\{(0,p,q),(p,0,q)\right\},\left\{(p,q,0),(0,q,p)\right\}$, and $\left\{(q,0,p),(q,p,0)\right\}$ for all $p, q$ with $0>p>q$ and $p+q\in \{-N-1,-N-2\}$;
            \item $\left\{(0, p, p),(p, 0, p),(p, p, 0)\right\}$ for $2p\in \{-N-1,-N-2\}$;
            \item $\left\{(0, 0, q)\right\}, \left\{(0, q, 0)\right\}$, and $\left\{(q, 0, 0)\right\}$ for $q\leq-N-1$.
        \end{itemize}
    \end{itemize}
    
    \begin{figure}[h!]
        \centering
        \includegraphics[width=10cm]{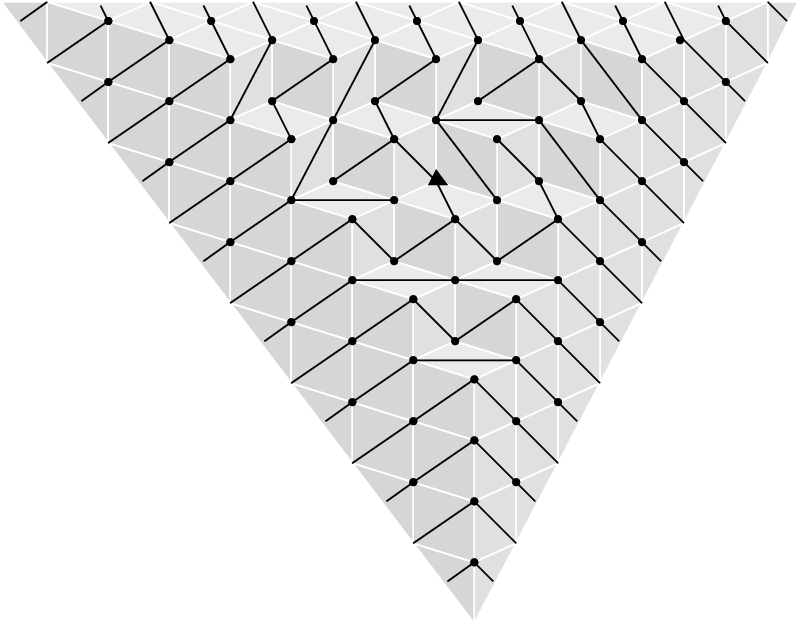}
        \caption{Example of a grove}
        \label{Figure 0}
        \floatfoot{Source: \cite{carroll}}
    \end{figure}
    
    \justify
    Carroll and Speyer also proved a bijection between groves and simplified groves. A simplified grove within radius $N$, where $N$ is a cutoff for $\mathcal{I}$ and furthermore is odd, is a subgraph $G'$ of $\mathcal{G}$ satisfying:
    
    \begin{itemize}
        \item (Vertex set) the vertex set of $G'$ is $\{(i,j,k)\in\mathcal{I} \text{ | } i+j+k\equiv0 \text{ mod } 2; i+j+k\geq-N-1\}$;
        \item (Acyclicity) $G'$ is acyclic;
        \item (Connectivity) every component of $G'$ contains exactly one of the following sets of vertices, and conversely, each such set is contained in some component:
        \begin{itemize}
            \item $\left\{(0,p,q),(p,0,q)\right\},\left\{(p,q,0),(0,q,p)\right\}$, and $\left\{(q,0,p),(q,p,0)\right\}$ for $p, q$ with $0>p>q$ and $p+q=-N-1$;
            \item $\left\{(0, \frac{-N-1}{2}, \frac{-N-1}{2}),(\frac{-N-1}{2}, 0, \frac{-N-1}{2}),(\frac{-N-1}{2}, \frac{-N-1}{2}, 0)\right\}$;
            \item $\left\{(0, 0, -N-1)\right\}, \left\{(0, -N-1, 0)\right\}$, and $\left\{(-N-1, 0, 0)\right\}$.
        \end{itemize}
    \end{itemize}
    
    \begin{figure}[h!]
        \centering
        \includegraphics[width=12cm]{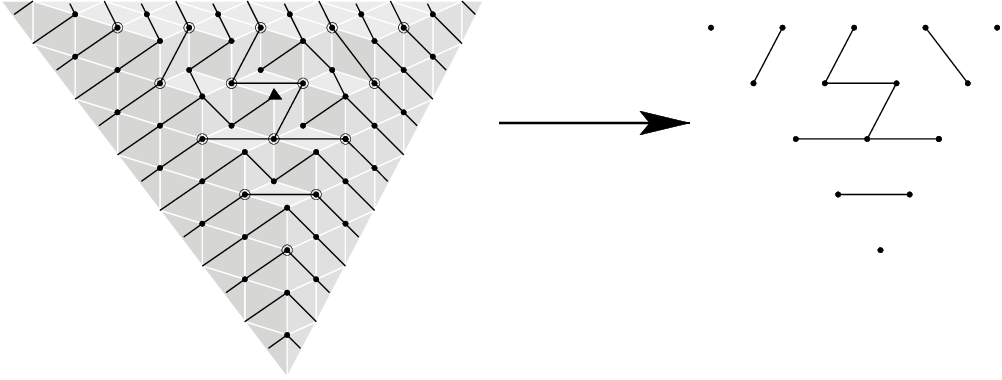}
        \caption{Example of a simplified grove}
        \label{Figure 1}
        \floatfoot{Source: \cite{carroll}}
    \end{figure}

	\justify
	\section{Definition}
	
	For our convenience, we will redefine a \textit{simplified grove of size $n$} to be a graph $G$ satisfying:
	
	\begin{itemize}
		\item (Vertex set) the vertex set of $G$ is $\{{(i,j)\in\mathbb{Z}^2\mid \lvert i+j\rvert\le n,j\leq 0,i+j\equiv n\bmod{2}}\}$;
		\item (Acyclicity) $G$ is acyclic;
		\item (Connectivity) the boundary vertices can be partitioned into the following sets so that each component of $G$ contains exactly one set, and conversely, each set is contained in some component:
		\begin{itemize}
		    \item $\{(-n,0)\},\{(n,0)\}$ and $\{(0,-n)\}$,
		    \item (west pairs) $\{(-i,0),(\frac{-n-i}{2},\frac{-n+i}{2})\}$ for $0<i<n, i\equiv n\bmod{2}$,
		    \item (east pairs) $\{(i,0),(\frac{n+i}{2},\frac{-n+i}{2})\}$ for $0<i<n, i\equiv n\bmod{2}$,
		    \item (south pairs) $\{(-n+i,-i),(n-i,-i\}$ for $\frac{n}{2}<i<n$,
		    \item (middle triplet) $\{(0,0),(-\frac{n}{2},-\frac{n}{2}),(\frac{n}{2},-\frac{n}{2})\}$ if $n$ is even.
		\end{itemize}
	\end{itemize}

	It can be checked that this new definition gives the same set of simplified groves as Carroll's. We also define a \textit{upward triangle of size $i$} to be a triangle whose vertices are $(a,b),(a-i,b-i),(a+i,b-i)$. Similarly, define an \textit{downward triangle of size $i$} to be a triangle whose vertices are $(a,b),(a-i,b+i),(a+i,b+i)$. For simplicity, we will refer to downward and upward triangle of size $1$ as downward and upward triangle respectively.
	
	Now assign to each downward triangle a number $1-e$ where $e$ is the number of edge in that triangle. From the acyclicity condition, we have the number of edge in every triangle is less than $3$; hence, the number assigned to each downward triangle is either $-1,0$ or $1$. Define an \textit{alternating sign triangle} to be a configuration of numbers generated by our process. It can be checked that this definition gives the same set of alternating sign triangles as Carroll's proposed definition does.
    
    \begin{figure}[h!]
        \centering
    
    \begin{tikzpicture}
        \filldraw[black] (0,0) circle (2pt);
        \filldraw[black] (1,0) circle (2pt);
        \filldraw[black] (2,0) circle (2pt);
        \filldraw[black] (3,0) circle (2pt);
        \filldraw[black] (4,0) circle (2pt);
        \filldraw[black] (0.5,-1) circle (2pt);
        \filldraw[black] (1.5,-1) circle (2pt);
        \filldraw[black] (2.5,-1) circle (2pt);
        \filldraw[black] (3.5,-1) circle (2pt);
        \filldraw[black] (1,-2) circle (2pt);
        \filldraw[black] (2,-2) circle (2pt);
        \filldraw[black] (3,-2) circle (2pt);
        \filldraw[black] (1.5,-3) circle (2pt);
        \filldraw[black] (2.5,-3) circle (2pt);
        \filldraw[black] (2,-4) circle (2pt);
        
        \draw[black, thick] (1,0) -- (0.5,-1);
        \draw[black, thick] (2,0) -- (2.5,-1);
        \draw[black, thick] (2.5,-1) -- (1.5,-1);
        \draw[black, thick] (1.5,-1) -- (1,-2);
        \draw[black, thick] (1.5,-1) -- (2,-2);
        \draw[black, thick] (2,-2) -- (3,-2);
        \draw[black, thick] (3,0) -- (3.5,-1);
        \draw[black, thick] (1.5,-3) -- (2.5,-3);
        
        \draw [-latex, thick](6,-2) -- (8,-2);
        
        \draw (10,-0.5) node[anchor=center] {0};
        \draw (11,-0.5) node[anchor=center] {1};
        \draw (12,-0.5) node[anchor=center] {0};
        \draw (13,-0.5) node[anchor=center] {0};
        \draw (10.5,-1.5) node[anchor=center] {0};
        \draw (11.5,-1.5) node[anchor=center] {-1};
        \draw (12.5,-1.5) node[anchor=center] {1};
        \draw (11,-2.5) node[anchor=center] {1};
        \draw (12,-2.5) node[anchor=center] {0};
        \draw (11.5,-3.5) node[anchor=center] {0};
        
    \end{tikzpicture}
        \caption{Getting alternating sign triangle of size 4 from a grove of size 4}
        \label{Figure 2}
    \end{figure}
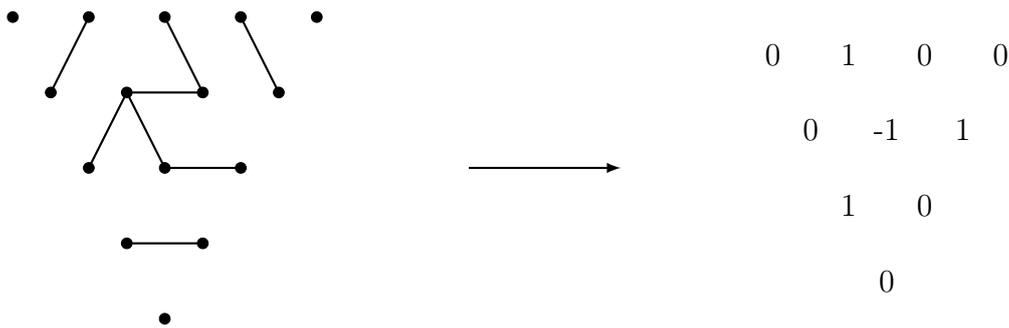
	
	\section{Permutation Triangle}
	
	Define \textit{permutation triangles} to be alternating sign triangles whose entries are either $0$ or $1$. We will prove some characteristics of permutation triangles. The proof will need one preliminary result:
	
	\begin{lemma}
		The sum of all entries in every alternating sign triangle of size $n$ is exactly $\lfloor\frac{n+1}{2}\rfloor$.
	\end{lemma}

	\begin{proof}
		It can easily be seen that an alternating sign triangle of size $n$ is taken from a grove of size $n$ whose vertices are divided into $3+\lfloor\frac{3n-3}{2}\rfloor$ components. In each component, the number of edges is fewer than the number of vertices by $1$. Hence, the number of edges in a grove of size $n$ is $\frac{(n+1)(n+2)}{2}-3-\lfloor\frac{3n-3}{2}\rfloor$.\\
		Notice that in each downward triangle, the entry is defined as $1-e$, where $e$ is the number of edges, and every edge belongs to exactly one downward triangle. Therefore, the sum of all entries is exactly the difference between the number of downward triangles and the number of edges, which is
		
		\begin{align*}
		\frac{(n)(n+1)}{2}-\left(\frac{(n+1)(n+2)}{2}-3-\Big\lfloor\frac{3n-3}{2}\Big\rfloor\right)=\Big\lfloor\frac{n+1}{2}\Big\rfloor.
		\end{align*}
	\end{proof}

    \justify
	Now we are prepared for our first theorem which was conjectured by Max Glick in 2013.
	
	\begin{theorem}
		A configuration is a permutation triangle if and only if the following properties are satisfied:
		\begin{enumerate}
		    \item All entries are either $1,0$ or $-1$,
		    \item The sum of all entries is $\Big\lfloor\frac{n+1}{2}\Big\rfloor$,
			\item The top $i$ rows have at most $i$ $1$-s,
			\item The left-most $i$ columns have at most $i$ $1$-s,
			\item The right-most $i$ columns have at most $i$ $1$-s,
			\item Any upward triangle of size $i$ has at most $i$ $1$-s.
		\end{enumerate}
	\end{theorem}

	\begin{proof}
		By \textit{lemma 1}, all permutation triangle have property $2$. Now, we will prove that every permutation triangle has property 3. Note that the sum of all entries is exactly $\lfloor\frac{n+1}{2}\rfloor$; therefore, we only need to prove this property for $i<\lfloor\frac{n+1}{2}\rfloor$.
		
		Consider a sub-graph whose vertices are those in the first $i+1$ rows of the grove and edges are those connecting such vertices, excluding edges connecting the vertices in row $i+1$. The connectivity conditions state that the vertices $(-n+2i,0),(-n+2i+2,0),...,(n-2i,0)$ have to be connected with some vertices in row $i+1$ and below. Hence, these vertices have to be connected with at least one vertex in row $i+1$. Therefore, at least $n+1-2i$ vertices in row $i+1$ have to be connected with a vertex in row $1$. Let $A$ be the set of such vertices, and $B$ be the set of vertices in row $i+1$ that does not belong to $A$. Since $\lvert A\rvert\geq n+1-2i$, and there are $n+1-i$ vertices in row $i+1$, $\lvert B\rvert \leq i$.
    
    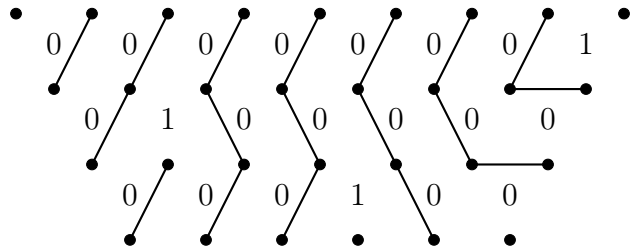
\begin{figure}[h!]
        \centering
    
    \begin{tikzpicture}
        \filldraw[black] (0,0) circle (2pt);
        \filldraw[black] (1,0) circle (2pt);
        \filldraw[black] (2,0) circle (2pt);
        \filldraw[black] (3,0) circle (2pt);
        \filldraw[black] (4,0) circle (2pt);
        \filldraw[black] (5,0) circle (2pt);
        \filldraw[black] (6,0) circle (2pt);
        \filldraw[black] (7,0) circle (2pt);
        \filldraw[black] (8,0) circle (2pt);
        \filldraw[black] (0.5,-1) circle (2pt);
        \filldraw[black] (1.5,-1) circle (2pt);
        \filldraw[black] (2.5,-1) circle (2pt);
        \filldraw[black] (3.5,-1) circle (2pt);
        \filldraw[black] (4.5,-1) circle (2pt);
        \filldraw[black] (5.5,-1) circle (2pt);
        \filldraw[black] (6.5,-1) circle (2pt);
        \filldraw[black] (7.5,-1) circle (2pt);
        \filldraw[black] (1,-2) circle (2pt);
        \filldraw[black] (2,-2) circle (2pt);
        \filldraw[black] (3,-2) circle (2pt);
        \filldraw[black] (4,-2) circle (2pt);
        \filldraw[black] (5,-2) circle (2pt);
        \filldraw[black] (6,-2) circle (2pt);
        \filldraw[black] (7,-2) circle (2pt);
        \filldraw[black] (1.5,-3) circle (2pt);
        \filldraw[black] (2.5,-3) circle (2pt);
        \filldraw[black] (3.5,-3) circle (2pt);
        \filldraw[black] (4.5,-3) circle (2pt);
        \filldraw[black] (5.5,-3) circle (2pt);
        \filldraw[black] (6.5,-3) circle (2pt);
        
        \draw[black, thick] (1,0) -- (0.5,-1);
        \draw[black, thick] (2,0) -- (1,-2);
        \draw[black, thick] (3,0) -- (2.5,-1);
        \draw[black, thick] (4,0) -- (3.5,-1);
        \draw[black, thick] (5,0) -- (4.5,-1);
        \draw[black, thick] (6,0) -- (5.5,-1);
        \draw[black, thick] (7,0) -- (6.5,-1);
        \draw[black, thick] (2.5,-1) -- (3,-2);
        \draw[black, thick] (3.5,-1) -- (4,-2);
        \draw[black, thick] (4.5,-1) -- (5.5,-3);
        \draw[black, thick] (5.5,-1) -- (6,-2);
        \draw[black, thick] (6.5,-1) -- (7.5,-1);
        \draw[black, thick] (2,-2) -- (1.5,-3);
        \draw[black, thick] (3,-2) -- (2.5,-3);
        \draw[black, thick] (4,-2) -- (3.5,-3);
        \draw[black, thick] (6,-2) -- (7,-2);
        
        \draw (0.5,-0.4) node[anchor=center] {0};
        \draw (1.5,-0.4) node[anchor=center] {0};
        \draw (2.5,-0.4) node[anchor=center] {0};
        \draw (3.5,-0.4) node[anchor=center] {0};
        \draw (4.5,-0.4) node[anchor=center] {0};
        \draw (5.5,-0.4) node[anchor=center] {0};
        \draw (6.5,-0.4) node[anchor=center] {0};
        \draw (7.5,-0.4) node[anchor=center] {1};
        \draw (1,-1.4) node[anchor=center] {0};
        \draw (2,-1.4) node[anchor=center] {1};
        \draw (3,-1.4) node[anchor=center] {0};
        \draw (4,-1.4) node[anchor=center] {0};
        \draw (5,-1.4) node[anchor=center] {0};
        \draw (6,-1.4) node[anchor=center] {0};
        \draw (7,-1.4) node[anchor=center] {0};
        \draw (1.5,-2.4) node[anchor=center] {0};
        \draw (2.5,-2.4) node[anchor=center] {0};
        \draw (3.5,-2.4) node[anchor=center] {0};
        \draw (4.5,-2.4) node[anchor=center] {1};
        \draw (5.5,-2.4) node[anchor=center] {0};
        \draw (6.5,-2.4) node[anchor=center] {0};
        
    \end{tikzpicture}
        \caption{Example with $i=3$}
        \label{Figure 3}
    \end{figure}
		
		The connectivity conditions also state that every vertex in the first $i$ rows has to be connected with either a vertex in row $1$, or a boundary vertex in row $i+1$ or below. Therefore, in our sub-graph, every vertex has to be connected with either a vertex in row $1$ or a vertex in $B$. This means that every components in the sub-graph has to contain at least a vertex in row $i$ or a vertex in $B$. Hence, we have at most $n+1+i$ components. Therefore, the minimum number of edges in the sub-graph is:
		
		\begin{align*}
		    \frac{(2n+2-i)(i+1)}{2}-(n+1+i)=ni+\frac{-i^2-i}{2}
		\end{align*}
		
		\justify
		On the other hand, the number of downward triangles in the sub-graph is:
		
		\begin{align*}
		    \frac{(n+n+1-i)(i)}{2}=ni+\frac{-i^2+i}{2}
		\end{align*}
		
		\justify
		Since each $0$-triangle gives $1$ edge while each $1$-triangle gives $0$ edge, the maximum number of $1$ is
		
		\begin{align*}
		    ni+\frac{-i^2+i}{2}-\left(ni+\frac{-i^2-i}{2}\right)=i
		\end{align*}
		
		\justify
		
		This completes our proof for property $3$. Due to symmetry, property $4$ and $5$ can be proved in the same way. Now we prove property $6$. Consider an upward triangle of size $i$, and let $C$ be the set of vertices inside the triangle. We have $\lvert C\rvert=\frac{i(i-1)}{2}$. Assume that the vertices in $C$ are divided into $k$ components, then the number of edges connecting the vertices in $C$ is $\frac{i(i-1)}{2}-k$. Since each component need to be connected with at least $1$ boundary vertex, we need at least $k$ more edges. Therefore, in the triangle, there are at least $\frac{i(i-1)}{2}$ edges. On the other hand, there are $\frac{i(i+1)}{2}$ downward triangles in the triangle. Hence, the maximum number of $1$ is $\frac{i(i+1)}{2}-\frac{i(i-1)}{2}=i$. This completes our proof for property $6$.
    
    \begin{figure}[h!]
        \centering
    
    \begin{tikzpicture}
        \filldraw[black] (3,0) circle (2pt);
        \filldraw[black] (4,0) circle (2pt);
        \filldraw[black] (2.5,-1) circle (2pt);
        \filldraw[black] (3.5,-1) circle (2pt);
        \filldraw[black] (4.5,-1) circle (2pt);
        \filldraw[black] (2,-2) circle (2pt);
        \filldraw[black] (3,-2) circle (2pt);
        \filldraw[black] (4,-2) circle (2pt);
        \filldraw[black] (5,-2) circle (2pt);
        \filldraw[black] (1.5,-3) circle (2pt);
        \filldraw[black] (2.5,-3) circle (2pt);
        \filldraw[black] (3.5,-3) circle (2pt);
        \filldraw[black] (4.5,-3) circle (2pt);
        \filldraw[black] (5.5,-3) circle (2pt);
        \filldraw[black] (2,-4) circle (2pt);
        \filldraw[black] (3,-4) circle (2pt);
        \filldraw[black] (4,-4) circle (2pt);
        \filldraw[black] (5,-4) circle (2pt);
        
        \draw[black, thick] (4,0) -- (3.5,-1);
        \draw[black, thick] (3.5,-1) -- (2.5,-1);
        \draw[black, thick] (3,-2) -- (4,-2);
        \draw[black, thick] (4.5,-3) -- (4,-2);
        \draw[black, thick] (1.5,-3) -- (2.5,-3);
        \draw[black, thick] (3.5,-3) -- (3,-4);
        \draw[black, thick] (4.5,-3) -- (5.5,-3);
        
        \draw (3.5,-0.4) node[anchor=center] {0};
        \draw (3,-1.4) node[anchor=center] {0};
        \draw (4,-1.4) node[anchor=center] {1};
        \draw (2.5,-2.4) node[anchor=center] {1};
        \draw (3.5,-2.4) node[anchor=center] {0};
        \draw (4.5,-2.4) node[anchor=center] {0};
        \draw (2,-3.4) node[anchor=center] {0};
        \draw (3,-3.4) node[anchor=center] {0};
        \draw (4,-3.4) node[anchor=center] {1};
        \draw (5,-3.4) node[anchor=center] {0};
        
    \end{tikzpicture}
        \caption{Example with $i=4$}
        \label{Figure 4}
    \end{figure}
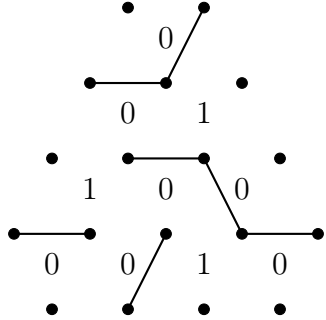
		
		\justify
		Now we will propose a method to construct a grove of size $n-1$ from a configuration of size $n-1$ that satisfies the six properties. First, we will assign our configuration to the upward triangles of a grove of size $n$. For each upward triangle that contains a $1$, we draw all three edges of that triangle. Now we connect the west pairs using the following procedure: assume that we can connect the pair $\{(-i,0),(\frac{-n-i}{2},\frac{-n+i}{2})\}$, \\
		\textit{step 1:} if there are two vertices $(a,b)$ and $(a-1,b-1)$ that are both connected with the pair $\{(-i,0),(\frac{-n-i}{2},\frac{-n+i}{2})\}$ and are not neighbors, we connect one of them with $(a+1,b-1)$. We repeat step 1 until no such pair exists, \\
		\textit{step 2:} if there is a vertex $(a,b)$ and some $i,j$ such that $(a,b)$ is not connected with the pair $\{(-i,0),(\frac{-n-i}{2},\frac{-n+i}{2})\}$ but $(a-i,b)$ and $(a+j,b)$ are, we connect $(a,b)$ with $(a+1,b+1)$; similarly, if there is a vertex $(a,b)$ and some $i,j$ such that $(a,b)$ is not connected with the pair $\{(-i,0),(\frac{-n-i}{2},\frac{-n+i}{2})\}$ but $(a-i,b+i)$ and $(a+j,b-j)$ are, we connect $(a,b)$ with $(a-1,b-1)$. We repeat step 2 until no such vertex exists, \\
		\textit{step 3:} we connect the pair $\{(-i+2,0),(\frac{-n-i+2}{2},\frac{-n+i-2}{2})\}$ such that there is no unconnected vertex in the region between the path and the component of the pair $\{(-i,0),(\frac{-n-i}{2},\frac{-n+i}{2})\}$. We will prove that this can be done.
		
	\begin{figure}[h!]
		
     \centering
        \begin{subfigure}[b]{0.2\textwidth}
            \centering
    
    \resizebox{\textwidth}{!}{
    \begin{tikzpicture}
        \filldraw[black] (0,0) circle (2pt);
        \filldraw[black] (1,0) circle (2pt);
        \filldraw[black] (2,0) circle (2pt);
        \filldraw[black] (3,0) circle (2pt);
        \filldraw[black] (4,0) circle (2pt);
        \filldraw[black] (5,0) circle (2pt);
        \filldraw[black] (6,0) circle (2pt);
        \filldraw[black] (7,0) circle (2pt);
        \filldraw[black] (0.5,-1) circle (2pt);
        \filldraw[black] (1.5,-1) circle (2pt);
        \filldraw[black] (2.5,-1) circle (2pt);
        \filldraw[black] (3.5,-1) circle (2pt);
        \filldraw[black] (4.5,-1) circle (2pt);
        \filldraw[black] (5.5,-1) circle (2pt);
        \filldraw[black] (6.5,-1) circle (2pt);
        \filldraw[black] (1,-2) circle (2pt);
        \filldraw[black] (2,-2) circle (2pt);
        \filldraw[black] (3,-2) circle (2pt);
        \filldraw[black] (4,-2) circle (2pt);
        \filldraw[black] (5,-2) circle (2pt);
        \filldraw[black] (6,-2) circle (2pt);
        \filldraw[black] (1.5,-3) circle (2pt);
        \filldraw[black] (2.5,-3) circle (2pt);
        \filldraw[black] (3.5,-3) circle (2pt);
        \filldraw[black] (4.5,-3) circle (2pt);
        \filldraw[black] (5.5,-3) circle (2pt);
        \filldraw[black] (2,-4) circle (2pt);
        \filldraw[black] (3,-4) circle (2pt);
        \filldraw[black] (4,-4) circle (2pt);
        \filldraw[black] (5,-4) circle (2pt);
        \filldraw[black] (2.5,-5) circle (2pt);
        \filldraw[black] (3.5,-5) circle (2pt);
        \filldraw[black] (4.5,-5) circle (2pt);
        \filldraw[black] (3,-6) circle (2pt);
        \filldraw[black] (4,-6) circle (2pt);
        \filldraw[black] (3.5,-7) circle (2pt);
        
        \draw[black, thick] (1,0) -- (0.5,-1);
        \draw[black, thick] (2,0) -- (1,-2);
        \draw[black, thick] (3,0) -- (1.5,-3);
        \draw[black, thick] (3,0) -- (3.5,-1);
        \draw[black, thick] (2.5,-1) -- (3.5,-1);
        \draw[black, thick] (2.5,-1) -- (3,-2);
        \draw[black, thick] (2,-2) -- (3,-2);
        \draw[black, thick] (4,-2) -- (3.5,-3);
        \draw[black, thick] (4,-2) -- (4.5,-3);
        \draw[black, thick] (4.5,-3) -- (3.5,-3);
        
        \draw (1,-0.6) node[anchor=center] {0};
        \draw (2,-0.6) node[anchor=center] {0};
        \draw (3,-0.6) node[anchor=center] {1};
        \draw (4,-0.6) node[anchor=center] {0};
        \draw (5,-0.6) node[anchor=center] {0};
        \draw (6,-0.6) node[anchor=center] {0};
        \draw (1.5,-1.6) node[anchor=center] {0};
        \draw (2.5,-1.6) node[anchor=center] {1};
        \draw (3.5,-1.6) node[anchor=center] {0};
        \draw (4.5,-1.6) node[anchor=center] {0};
        \draw (5.5,-1.6) node[anchor=center] {0};
        \draw (2,-2.6) node[anchor=center] {0};
        \draw (3,-2.6) node[anchor=center] {0};
        \draw (4,-2.6) node[anchor=center] {1};
        \draw (5,-2.6) node[anchor=center] {0};
        \draw (2.5,-3.6) node[anchor=center] {0};
        \draw (3.5,-3.6) node[anchor=center] {0};
        \draw (4.5,-3.6) node[anchor=center] {0};
        \draw (3,-4.6) node[anchor=center] {0};
        \draw (4,-4.6) node[anchor=center] {0};
        \draw (3.5,-5.6) node[anchor=center] {0};
        
    \end{tikzpicture}
    }
            \caption{Before step 1}
        \end{subfigure}
     \hfill
        \begin{subfigure}[b]{0.2\textwidth}
            \centering
    
    \resizebox{\textwidth}{!}{
    \begin{tikzpicture}
        \filldraw[black] (0,0) circle (2pt);
        \filldraw[black] (1,0) circle (2pt);
        \filldraw[black] (2,0) circle (2pt);
        \filldraw[black] (3,0) circle (2pt);
        \filldraw[black] (4,0) circle (2pt);
        \filldraw[black] (5,0) circle (2pt);
        \filldraw[black] (6,0) circle (2pt);
        \filldraw[black] (7,0) circle (2pt);
        \filldraw[black] (0.5,-1) circle (2pt);
        \filldraw[black] (1.5,-1) circle (2pt);
        \filldraw[black] (2.5,-1) circle (2pt);
        \filldraw[black] (3.5,-1) circle (2pt);
        \filldraw[black] (4.5,-1) circle (2pt);
        \filldraw[black] (5.5,-1) circle (2pt);
        \filldraw[black] (6.5,-1) circle (2pt);
        \filldraw[black] (1,-2) circle (2pt);
        \filldraw[black] (2,-2) circle (2pt);
        \filldraw[black] (3,-2) circle (2pt);
        \filldraw[black] (4,-2) circle (2pt);
        \filldraw[black] (5,-2) circle (2pt);
        \filldraw[black] (6,-2) circle (2pt);
        \filldraw[black] (1.5,-3) circle (2pt);
        \filldraw[black] (2.5,-3) circle (2pt);
        \filldraw[black] (3.5,-3) circle (2pt);
        \filldraw[black] (4.5,-3) circle (2pt);
        \filldraw[black] (5.5,-3) circle (2pt);
        \filldraw[black] (2,-4) circle (2pt);
        \filldraw[black] (3,-4) circle (2pt);
        \filldraw[black] (4,-4) circle (2pt);
        \filldraw[black] (5,-4) circle (2pt);
        \filldraw[black] (2.5,-5) circle (2pt);
        \filldraw[black] (3.5,-5) circle (2pt);
        \filldraw[black] (4.5,-5) circle (2pt);
        \filldraw[black] (3,-6) circle (2pt);
        \filldraw[black] (4,-6) circle (2pt);
        \filldraw[black] (3.5,-7) circle (2pt);
        
        \draw[black, thick] (1,0) -- (0.5,-1);
        \draw[black, thick] (2,0) -- (1,-2);
        \draw[black, thick] (3,0) -- (1.5,-3);
        \draw[black, thick] (3,0) -- (3.5,-1);
        \draw[black, thick] (4,-2) -- (3.5,-1);
        \draw[black, thick] (2.5,-1) -- (3.5,-1);
        \draw[black, thick] (2.5,-1) -- (3,-2);
        \draw[black, thick] (2,-2) -- (3,-2);
        \draw[black, thick] (4,-2) -- (3.5,-3);
        \draw[black, thick] (4,-2) -- (4.5,-3);
        \draw[black, thick] (4.5,-3) -- (3.5,-3);
        
        \draw (1,-0.6) node[anchor=center] {0};
        \draw (2,-0.6) node[anchor=center] {0};
        \draw (3,-0.6) node[anchor=center] {1};
        \draw (4,-0.6) node[anchor=center] {0};
        \draw (5,-0.6) node[anchor=center] {0};
        \draw (6,-0.6) node[anchor=center] {0};
        \draw (1.5,-1.6) node[anchor=center] {0};
        \draw (2.5,-1.6) node[anchor=center] {1};
        \draw (3.5,-1.6) node[anchor=center] {0};
        \draw (4.5,-1.6) node[anchor=center] {0};
        \draw (5.5,-1.6) node[anchor=center] {0};
        \draw (2,-2.6) node[anchor=center] {0};
        \draw (3,-2.6) node[anchor=center] {0};
        \draw (4,-2.6) node[anchor=center] {1};
        \draw (5,-2.6) node[anchor=center] {0};
        \draw (2.5,-3.6) node[anchor=center] {0};
        \draw (3.5,-3.6) node[anchor=center] {0};
        \draw (4.5,-3.6) node[anchor=center] {0};
        \draw (3,-4.6) node[anchor=center] {0};
        \draw (4,-4.6) node[anchor=center] {0};
        \draw (3.5,-5.6) node[anchor=center] {0};
        
    \end{tikzpicture}
    }
            \caption{After step 1}
        \end{subfigure}
     \hfill
        \begin{subfigure}[b]{0.2\textwidth}
            \centering
    
    \resizebox{\textwidth}{!}{
    \begin{tikzpicture}
        \filldraw[black] (0,0) circle (2pt);
        \filldraw[black] (1,0) circle (2pt);
        \filldraw[black] (2,0) circle (2pt);
        \filldraw[black] (3,0) circle (2pt);
        \filldraw[black] (4,0) circle (2pt);
        \filldraw[black] (5,0) circle (2pt);
        \filldraw[black] (6,0) circle (2pt);
        \filldraw[black] (7,0) circle (2pt);
        \filldraw[black] (0.5,-1) circle (2pt);
        \filldraw[black] (1.5,-1) circle (2pt);
        \filldraw[black] (2.5,-1) circle (2pt);
        \filldraw[black] (3.5,-1) circle (2pt);
        \filldraw[black] (4.5,-1) circle (2pt);
        \filldraw[black] (5.5,-1) circle (2pt);
        \filldraw[black] (6.5,-1) circle (2pt);
        \filldraw[black] (1,-2) circle (2pt);
        \filldraw[black] (2,-2) circle (2pt);
        \filldraw[black] (3,-2) circle (2pt);
        \filldraw[black] (4,-2) circle (2pt);
        \filldraw[black] (5,-2) circle (2pt);
        \filldraw[black] (6,-2) circle (2pt);
        \filldraw[black] (1.5,-3) circle (2pt);
        \filldraw[black] (2.5,-3) circle (2pt);
        \filldraw[black] (3.5,-3) circle (2pt);
        \filldraw[black] (4.5,-3) circle (2pt);
        \filldraw[black] (5.5,-3) circle (2pt);
        \filldraw[black] (2,-4) circle (2pt);
        \filldraw[black] (3,-4) circle (2pt);
        \filldraw[black] (4,-4) circle (2pt);
        \filldraw[black] (5,-4) circle (2pt);
        \filldraw[black] (2.5,-5) circle (2pt);
        \filldraw[black] (3.5,-5) circle (2pt);
        \filldraw[black] (4.5,-5) circle (2pt);
        \filldraw[black] (3,-6) circle (2pt);
        \filldraw[black] (4,-6) circle (2pt);
        \filldraw[black] (3.5,-7) circle (2pt);
        
        \draw[black, thick] (1,0) -- (0.5,-1);
        \draw[black, thick] (2,0) -- (1,-2);
        \draw[black, thick] (3,0) -- (1.5,-3);
        \draw[black, thick] (3,0) -- (3.5,-1);
        \draw[black, thick] (2.5,-1) -- (3.5,-1);
        \draw[black, thick] (2.5,-1) -- (3,-2);
        \draw[black, thick] (2,-2) -- (3,-2);
        \draw[black, thick] (4,-2) -- (3.5,-3);
        \draw[black, thick] (4,-2) -- (4.5,-3);
        \draw[black, thick] (4.5,-3) -- (3.5,-3);
        \draw[black, thick] (4,-2) -- (3.5,-1);
        \draw[black, thick] (3,-2) -- (2.5,-3);
        
        \draw (1,-0.6) node[anchor=center] {0};
        \draw (2,-0.6) node[anchor=center] {0};
        \draw (3,-0.6) node[anchor=center] {1};
        \draw (4,-0.6) node[anchor=center] {0};
        \draw (5,-0.6) node[anchor=center] {0};
        \draw (6,-0.6) node[anchor=center] {0};
        \draw (1.5,-1.6) node[anchor=center] {0};
        \draw (2.5,-1.6) node[anchor=center] {1};
        \draw (3.5,-1.6) node[anchor=center] {0};
        \draw (4.5,-1.6) node[anchor=center] {0};
        \draw (5.5,-1.6) node[anchor=center] {0};
        \draw (2,-2.6) node[anchor=center] {0};
        \draw (3,-2.6) node[anchor=center] {0};
        \draw (4,-2.6) node[anchor=center] {1};
        \draw (5,-2.6) node[anchor=center] {0};
        \draw (2.5,-3.6) node[anchor=center] {0};
        \draw (3.5,-3.6) node[anchor=center] {0};
        \draw (4.5,-3.6) node[anchor=center] {0};
        \draw (3,-4.6) node[anchor=center] {0};
        \draw (4,-4.6) node[anchor=center] {0};
        \draw (3.5,-5.6) node[anchor=center] {0};
        
    \end{tikzpicture}
    }
            \caption{After step 2}
        \end{subfigure}
     \hfill
        \begin{subfigure}[b]{0.2\textwidth}
            \centering
    
    \resizebox{\textwidth}{!}{
    \begin{tikzpicture}
        \filldraw[black] (0,0) circle (2pt);
        \filldraw[black] (1,0) circle (2pt);
        \filldraw[black] (2,0) circle (2pt);
        \filldraw[black] (3,0) circle (2pt);
        \filldraw[black] (4,0) circle (2pt);
        \filldraw[black] (5,0) circle (2pt);
        \filldraw[black] (6,0) circle (2pt);
        \filldraw[black] (7,0) circle (2pt);
        \filldraw[black] (0.5,-1) circle (2pt);
        \filldraw[black] (1.5,-1) circle (2pt);
        \filldraw[black] (2.5,-1) circle (2pt);
        \filldraw[black] (3.5,-1) circle (2pt);
        \filldraw[black] (4.5,-1) circle (2pt);
        \filldraw[black] (5.5,-1) circle (2pt);
        \filldraw[black] (6.5,-1) circle (2pt);
        \filldraw[black] (1,-2) circle (2pt);
        \filldraw[black] (2,-2) circle (2pt);
        \filldraw[black] (3,-2) circle (2pt);
        \filldraw[black] (4,-2) circle (2pt);
        \filldraw[black] (5,-2) circle (2pt);
        \filldraw[black] (6,-2) circle (2pt);
        \filldraw[black] (1.5,-3) circle (2pt);
        \filldraw[black] (2.5,-3) circle (2pt);
        \filldraw[black] (3.5,-3) circle (2pt);
        \filldraw[black] (4.5,-3) circle (2pt);
        \filldraw[black] (5.5,-3) circle (2pt);
        \filldraw[black] (2,-4) circle (2pt);
        \filldraw[black] (3,-4) circle (2pt);
        \filldraw[black] (4,-4) circle (2pt);
        \filldraw[black] (5,-4) circle (2pt);
        \filldraw[black] (2.5,-5) circle (2pt);
        \filldraw[black] (3.5,-5) circle (2pt);
        \filldraw[black] (4.5,-5) circle (2pt);
        \filldraw[black] (3,-6) circle (2pt);
        \filldraw[black] (4,-6) circle (2pt);
        \filldraw[black] (3.5,-7) circle (2pt);
        
        \draw[black, thick] (1,0) -- (0.5,-1);
        \draw[black, thick] (2,0) -- (1,-2);
        \draw[black, thick] (3,0) -- (1.5,-3);
        \draw[black, thick] (3,0) -- (3.5,-1);
        \draw[black, thick] (2.5,-1) -- (3.5,-1);
        \draw[black, thick] (2.5,-1) -- (3,-2);
        \draw[black, thick] (2,-2) -- (3,-2);
        \draw[black, thick] (4,-2) -- (3.5,-3);
        \draw[black, thick] (4,-2) -- (4.5,-3);
        \draw[black, thick] (4.5,-3) -- (3.5,-3);
        \draw[black, thick] (4,-2) -- (3.5,-1);
        \draw[black, thick] (3,-2) -- (2.5,-3);
        \draw[black, thick] (4,0) -- (5.5,-3);
        \draw[black, thick] (5,-4) -- (5.5,-3);
        \draw[black, thick] (5,-4) -- (2,-4);
        
        \draw (1,-0.6) node[anchor=center] {0};
        \draw (2,-0.6) node[anchor=center] {0};
        \draw (3,-0.6) node[anchor=center] {1};
        \draw (4,-0.6) node[anchor=center] {0};
        \draw (5,-0.6) node[anchor=center] {0};
        \draw (6,-0.6) node[anchor=center] {0};
        \draw (1.5,-1.6) node[anchor=center] {0};
        \draw (2.5,-1.6) node[anchor=center] {1};
        \draw (3.5,-1.6) node[anchor=center] {0};
        \draw (4.5,-1.6) node[anchor=center] {0};
        \draw (5.5,-1.6) node[anchor=center] {0};
        \draw (2,-2.6) node[anchor=center] {0};
        \draw (3,-2.6) node[anchor=center] {0};
        \draw (4,-2.6) node[anchor=center] {1};
        \draw (5,-2.6) node[anchor=center] {0};
        \draw (2.5,-3.6) node[anchor=center] {0};
        \draw (3.5,-3.6) node[anchor=center] {0};
        \draw (4.5,-3.6) node[anchor=center] {0};
        \draw (3,-4.6) node[anchor=center] {0};
        \draw (4,-4.6) node[anchor=center] {0};
        \draw (3.5,-5.6) node[anchor=center] {0};
        
    \end{tikzpicture}
    }
            \caption{After step 3}
        \end{subfigure}
        
        \caption{Three steps illustration}
        \label{Figure 5}
        
    \end{figure}
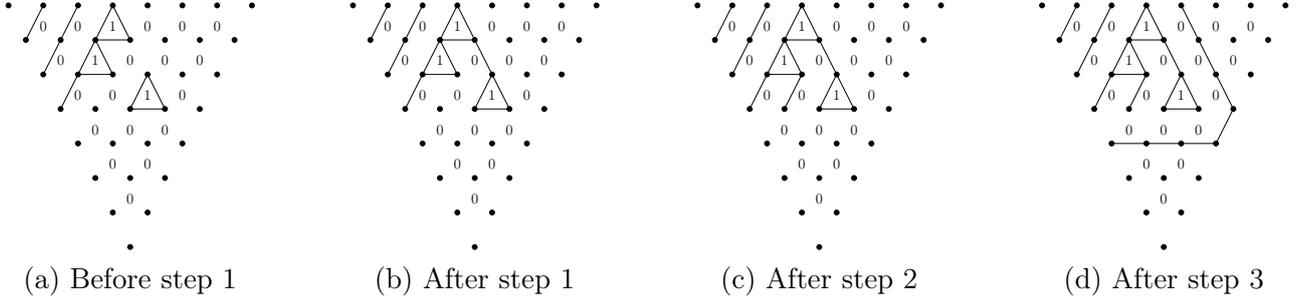
    
        Clearly, we can connect the pair $\{(-n+2,0),(-n+1,-1)\}$. Assume that we can connect the pair $\{(-i,0),(\frac{-n-i}{2},\frac{-n+i}{2})\}$, consider the vertices on the segment $s$ between the pair $\{(-i+2,0),(\frac{n-3i+4}{2},\frac{-n+i}{2})\}$.
        
        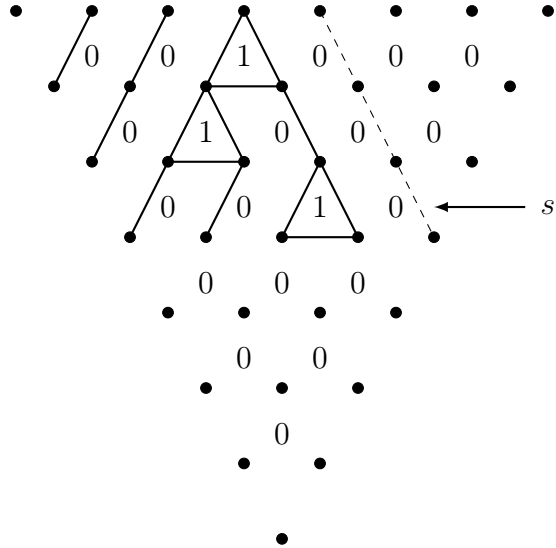
\begin{figure}[h!]
        \centering
    
    \begin{tikzpicture}
        \filldraw[black] (0,0) circle (2pt);
        \filldraw[black] (1,0) circle (2pt);
        \filldraw[black] (2,0) circle (2pt);
        \filldraw[black] (3,0) circle (2pt);
        \filldraw[black] (4,0) circle (2pt);
        \filldraw[black] (5,0) circle (2pt);
        \filldraw[black] (6,0) circle (2pt);
        \filldraw[black] (7,0) circle (2pt);
        \filldraw[black] (0.5,-1) circle (2pt);
        \filldraw[black] (1.5,-1) circle (2pt);
        \filldraw[black] (2.5,-1) circle (2pt);
        \filldraw[black] (3.5,-1) circle (2pt);
        \filldraw[black] (4.5,-1) circle (2pt);
        \filldraw[black] (5.5,-1) circle (2pt);
        \filldraw[black] (6.5,-1) circle (2pt);
        \filldraw[black] (1,-2) circle (2pt);
        \filldraw[black] (2,-2) circle (2pt);
        \filldraw[black] (3,-2) circle (2pt);
        \filldraw[black] (4,-2) circle (2pt);
        \filldraw[black] (5,-2) circle (2pt);
        \filldraw[black] (6,-2) circle (2pt);
        \filldraw[black] (1.5,-3) circle (2pt);
        \filldraw[black] (2.5,-3) circle (2pt);
        \filldraw[black] (3.5,-3) circle (2pt);
        \filldraw[black] (4.5,-3) circle (2pt);
        \filldraw[black] (5.5,-3) circle (2pt);
        \filldraw[black] (2,-4) circle (2pt);
        \filldraw[black] (3,-4) circle (2pt);
        \filldraw[black] (4,-4) circle (2pt);
        \filldraw[black] (5,-4) circle (2pt);
        \filldraw[black] (2.5,-5) circle (2pt);
        \filldraw[black] (3.5,-5) circle (2pt);
        \filldraw[black] (4.5,-5) circle (2pt);
        \filldraw[black] (3,-6) circle (2pt);
        \filldraw[black] (4,-6) circle (2pt);
        \filldraw[black] (3.5,-7) circle (2pt);
        
        \draw[black, thick] (1,0) -- (0.5,-1);
        \draw[black, thick] (2,0) -- (1,-2);
        \draw[black, thick] (3,0) -- (1.5,-3);
        \draw[black, thick] (3,0) -- (3.5,-1);
        \draw[black, thick] (2.5,-1) -- (3.5,-1);
        \draw[black, thick] (2.5,-1) -- (3,-2);
        \draw[black, thick] (2,-2) -- (3,-2);
        \draw[black, thick] (4,-2) -- (3.5,-3);
        \draw[black, thick] (4,-2) -- (4.5,-3);
        \draw[black, thick] (4.5,-3) -- (3.5,-3);
        \draw[black, thick] (4,-2) -- (3.5,-1);
        \draw[black, thick] (3,-2) -- (2.5,-3);
        
        \draw (1,-0.6) node[anchor=center] {0};
        \draw (2,-0.6) node[anchor=center] {0};
        \draw (3,-0.6) node[anchor=center] {1};
        \draw (4,-0.6) node[anchor=center] {0};
        \draw (5,-0.6) node[anchor=center] {0};
        \draw (6,-0.6) node[anchor=center] {0};
        \draw (1.5,-1.6) node[anchor=center] {0};
        \draw (2.5,-1.6) node[anchor=center] {1};
        \draw (3.5,-1.6) node[anchor=center] {0};
        \draw (4.5,-1.6) node[anchor=center] {0};
        \draw (5.5,-1.6) node[anchor=center] {0};
        \draw (2,-2.6) node[anchor=center] {0};
        \draw (3,-2.6) node[anchor=center] {0};
        \draw (4,-2.6) node[anchor=center] {1};
        \draw (5,-2.6) node[anchor=center] {0};
        \draw (2.5,-3.6) node[anchor=center] {0};
        \draw (3.5,-3.6) node[anchor=center] {0};
        \draw (4.5,-3.6) node[anchor=center] {0};
        \draw (3,-4.6) node[anchor=center] {0};
        \draw (4,-4.6) node[anchor=center] {0};
        \draw (3.5,-5.6) node[anchor=center] {0};

        \draw (7,-2.6) node[anchor=center] {$s$};
        \draw[black, dashed] (4,0) -- (5.5,-3);
        
        \draw [-latex, thick](6.7,-2.6) -- (5.5,-2.6);
        
    \end{tikzpicture}
        \caption{Segment $s$}
        \label{Figure 5b}
    \end{figure}
        
        If there are some vertices on $s$ that are connected to the pair $\{(-i,0),(\frac{-n-i}{2},\frac{-n+i}{2})\}$, let $k$ be the number of such vertices, and consider the region to the west of $s$. Also, let $l$ be the number of vertices in that region, but not on $s$, that are not connected to any pair in the region. Finally, for our convenience, let $a$ be the number of boundary pairs in the region. Then, the number of vertices connected to one of the boundary pairs in the region is
        
        $$(a+1)^{2}+k-l = a^2+2a+1+k-l$$
        
        \justify
        Now for a moment, assume that in each $1$ triangle, we only draw two out of three edges. This does not change the connectivity of our components; however, now, there is no cycle in our graph, which means the number of edges is exactly difference between the number of vertices and the number of components, which is
        
        $$a^2+2a+1+k-l-(a+1)=a^2+a+k-l$$
        
        \justify
        Now we count the number of triangles that have edges in the components. First, among the triangles that have no edge lying on $s$, since there are $l$ vertices that are not on $s$ and not connected to the pair $\{(-i,0),(\frac{-n-i}{2},\frac{-n+i}{2})\}$, at least $l$ of such triangles do not have edges in the components. Hence, among these triangles, there are at most $a^2-l$ triangles having edges in the components.
        
        Among the remaining triangles, since there are $k$ vertices on $s$ connected to the pair $\{(-i,0),(\frac{-n-i}{2},\frac{-n+i}{2})\}$, there are at most $k$ more triangles having edges in the components. However, consider the highest vertex on $s$ that are connected to the pair $\{(-i,0),(\frac{-n-i}{2},\frac{-n+i}{2})\}$, this vertex cannot be connected in \textit{step 1} or \textit{step 2}, it can only be connected through an $1$ triangle. This means that this vertex, and the vertex below it in $s$, are connected to the pair $\{(-i,0),(\frac{-n-i}{2},\frac{-n+i}{2})\}$ by the same triangle, which means there are only at most $k-1$ more triangles having edges in the components. Hence, in total, the maximum number of triangles having edges in the components is
        
        $$a^2-l+k-1$$
        
        \justify
        Since each edge is part of exactly $1$ upward triangle, and each $0$-triangle has exactly $1$ edge while each $1$-triangle has exactly $2$, the number of edges is the sum of the number of triangles and the number of $1$-triangles. Hence, the number of $1$-triangles is at least
        
        $$(a^2+a+k-l)-(a^2+k-l-1)=a+1$$
        
        \justify
        This is a contradiction since the triangles all lie on the first $a$ rows, which have at most $a$ $1$ triangles. Hence, there is no vertex connected to the pair $\{(-i,0),(\frac{-n-i}{2},\frac{-n+i}{2})\}$ on $s$. Similarly, there is no such vertex on the segment $s'$ between the pair $\{(\frac{-n-i+2}{2},\frac{-n+i-2}{2}),(\frac{n-3i+2}{2},\frac{-n+i-2}{2})\}$. Therefore, no matter how we connect the pair $\{(-i,0),(\frac{-n-i}{2},\frac{-n+i}{2})\}$, in the worst case, we can still connect the pair $\{(-i+2,0),(\frac{-n-i+2}{2},\frac{-n+i-2}{2})\}$ through $s$ and $s'$.
        
        Thus, we have proved that we can connect the west pairs. Similarly, we can connect the east and south pairs. Now we will prove that no two paths are connected.
        
        If $n$ is odd, and assume that the pairs $\{(-1,0),(\frac{-n-1}{2},\frac{-n+1}{2})\}$ and $\{(1,0),(\frac{n+1}{2},\frac{-n+1}{2})\}$ are connected, then \textit{step 2} assures that all vertices on the first $\frac{-n+1}{2}+1$ rows all belong to one of the west and east pairs. However, since the pairs $\{(-1,0),(\frac{-n-1}{2},\frac{-n+1}{2})\}$ and $\{(1,0),(\frac{n+1}{2},\frac{-n+1}{2})\}$ are connected, there are only at most $n$ components. Now again, for a moment, assume that in each $1$ triangle, we only draw two out of three edges, then the number of edges is at least $v-n$ edges where $v$ is the number of vertices on the first $\frac{-n+1}{2}+1$ rows. On the other hand, the number of triangles on the first $\frac{-n+1}{2}+1$ rows is $v-\left(\frac{n+1}{2}+1\right)-2\frac{n-1}{2}=v-n-\frac{n+1}{2}$. Hence, the number of $1$ triangles is at least
        
        $$(v-n) - \left(v-n-\frac{n+1}{2}\right) = \frac{n+1}{2}$$
        
        \justify
        which is a contradiction since these are all on the first $\frac{n-1}{2}$ rows. Therefore, the two pairs are not connected, and hence no two west and east pairs are connected. Similarly, no two west and south pairs or east and south pairs are connected. If $n$ is even, similarly, using the same process as above, we can prove that no two paths are connected, and we can also connect the middle triplet so that it is not connected with any path.
		
		Now mark the center of the downward triangles and color the new vertices red. We can see that the red vertices are divided into different regions by the edges we have drawn. Also, we will prove that every red vertex is in the same region with a boundary vertex. Assume there is a region that does not contain any boundary vertex, then this region has to be bounded by the edges we have drawn. If no boundary edge is connected to a boundary pair, then all boundary edges are edges of $1$-triangle, which contradicts property $6$. If any boundary edge is connected with a boundary pair, then using the same technique as above, we can prove that this contradicts one of the three properties $3,4,5$. Hence, every red vertex is in the same region with exactly one pair of boundary red vertices.
		
		Now we connect the red vertices by red edges so that each edge goes through exactly one upward triangle. This can be done easily for the west, east, and south triangles since \textit{step 1} and \textit{step 2} guarantees that each upward triangle has exactly two neighbour downward triangles in the same region, and we have proved that each downward triangle has at least one neighbour upward triangles in the same region. This also gives that in each west, east and south region, the number of upward triangles is fewer than that of downward triangles by one. Therefore, with some simple calculations, we can prove that in the middle region, the number of upward triangles is also fewer than that of downward triangles by one. We also have that each upward triangle has at least two neighbour downward triangles in the same region, and we have proved that each downward triangle has at least one neighbour upward triangles in the same region. Therefore, we can connect the red vertices in this region.
		
		Now we can see that the red vertices and edges form a grove of size $n-1$ and our configuration of numbers lie inside the downward triangles. Indeed, since each region contains exactly one boundary pair, each component of the new grove also contains exactly one boundary pair. Since each $0$-triangle has exactly one red edge going through it while each $1$-triangle has none, in the new grove, each $0$-triangle has exactly one edge it while each $1$-triangle has none. Last but not least, since in each region, every red vertices are connected using exactly one fewer edges (since there are one fewer upward triangle than downward triangle), the new grove is acyclic. Hence, we have finished constructing a grove from a configuration. This completes the proof.
		
	\end{proof}
		
	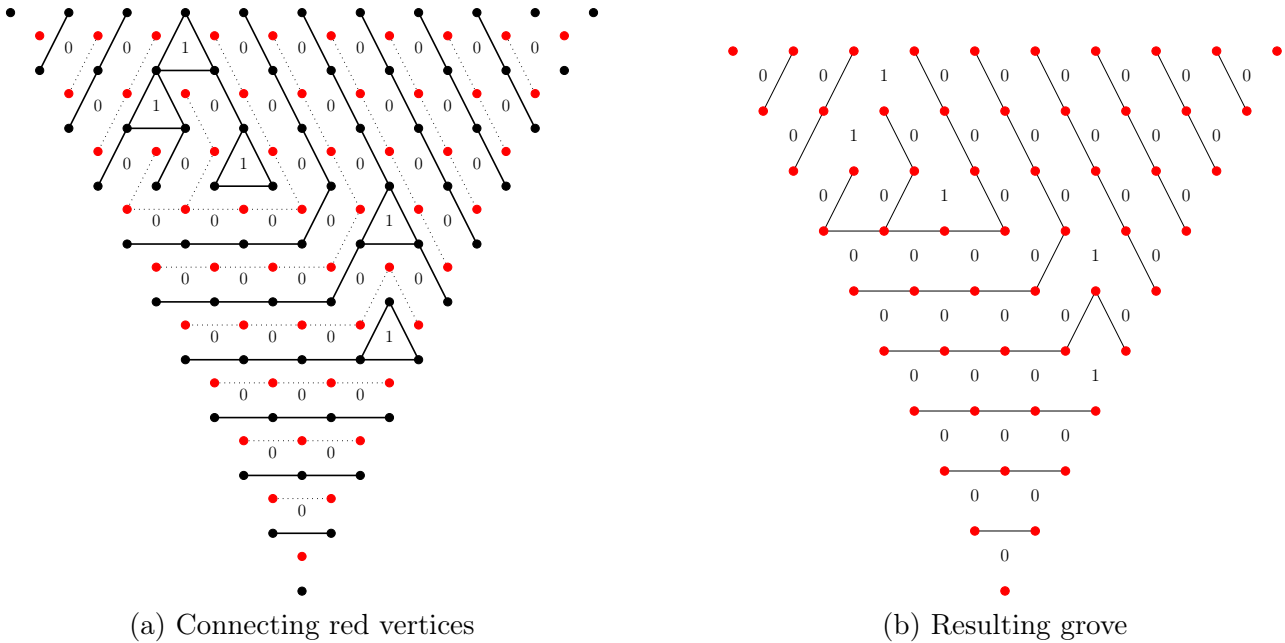
\begin{figure}[h!]
		
     \centering
        \begin{subfigure}[b]{8cm}
            \centering
    
    \resizebox{\textwidth}{!}{
    \begin{tikzpicture}
        \filldraw[black] (0,0) circle (2pt);
        \filldraw[black] (1,0) circle (2pt);
        \filldraw[black] (2,0) circle (2pt);
        \filldraw[black] (3,0) circle (2pt);
        \filldraw[black] (4,0) circle (2pt);
        \filldraw[black] (5,0) circle (2pt);
        \filldraw[black] (6,0) circle (2pt);
        \filldraw[black] (7,0) circle (2pt);
        \filldraw[black] (8,0) circle (2pt);
        \filldraw[black] (9,0) circle (2pt);
        \filldraw[black] (10,0) circle (2pt);
        \filldraw[black] (0.5,-1) circle (2pt);
        \filldraw[black] (1.5,-1) circle (2pt);
        \filldraw[black] (2.5,-1) circle (2pt);
        \filldraw[black] (3.5,-1) circle (2pt);
        \filldraw[black] (4.5,-1) circle (2pt);
        \filldraw[black] (5.5,-1) circle (2pt);
        \filldraw[black] (6.5,-1) circle (2pt);
        \filldraw[black] (7.5,-1) circle (2pt);
        \filldraw[black] (8.5,-1) circle (2pt);
        \filldraw[black] (9.5,-1) circle (2pt);
        \filldraw[black] (1,-2) circle (2pt);
        \filldraw[black] (2,-2) circle (2pt);
        \filldraw[black] (3,-2) circle (2pt);
        \filldraw[black] (4,-2) circle (2pt);
        \filldraw[black] (5,-2) circle (2pt);
        \filldraw[black] (6,-2) circle (2pt);
        \filldraw[black] (7,-2) circle (2pt);
        \filldraw[black] (8,-2) circle (2pt);
        \filldraw[black] (9,-2) circle (2pt);
        \filldraw[black] (1.5,-3) circle (2pt);
        \filldraw[black] (2.5,-3) circle (2pt);
        \filldraw[black] (3.5,-3) circle (2pt);
        \filldraw[black] (4.5,-3) circle (2pt);
        \filldraw[black] (5.5,-3) circle (2pt);
        \filldraw[black] (6.5,-3) circle (2pt);
        \filldraw[black] (7.5,-3) circle (2pt);
        \filldraw[black] (8.5,-3) circle (2pt);
        \filldraw[black] (2,-4) circle (2pt);
        \filldraw[black] (3,-4) circle (2pt);
        \filldraw[black] (4,-4) circle (2pt);
        \filldraw[black] (5,-4) circle (2pt);
        \filldraw[black] (6,-4) circle (2pt);
        \filldraw[black] (7,-4) circle (2pt);
        \filldraw[black] (8,-4) circle (2pt);
        \filldraw[black] (2.5,-5) circle (2pt);
        \filldraw[black] (3.5,-5) circle (2pt);
        \filldraw[black] (4.5,-5) circle (2pt);
        \filldraw[black] (5.5,-5) circle (2pt);
        \filldraw[black] (6.5,-5) circle (2pt);
        \filldraw[black] (7.5,-5) circle (2pt);
        \filldraw[black] (3,-6) circle (2pt);
        \filldraw[black] (4,-6) circle (2pt);
        \filldraw[black] (5,-6) circle (2pt);
        \filldraw[black] (6,-6) circle (2pt);
        \filldraw[black] (7,-6) circle (2pt);
        \filldraw[black] (3.5,-7) circle (2pt);
        \filldraw[black] (4.5,-7) circle (2pt);
        \filldraw[black] (5.5,-7) circle (2pt);
        \filldraw[black] (6.5,-7) circle (2pt);
        \filldraw[black] (4,-8) circle (2pt);
        \filldraw[black] (5,-8) circle (2pt);
        \filldraw[black] (6,-8) circle (2pt);
        \filldraw[black] (4.5,-9) circle (2pt);
        \filldraw[black] (5.5,-9) circle (2pt);
        \filldraw[black] (5,-10) circle (2pt);
        
        \draw[black, thick] (1,0) -- (0.5,-1);
        \draw[black, thick] (2,0) -- (1,-2);
        \draw[black, thick] (3,0) -- (1.5,-3);
        \draw[black, thick] (3,0) -- (3.5,-1);
        \draw[black, thick] (2.5,-1) -- (3.5,-1);
        \draw[black, thick] (2.5,-1) -- (3,-2);
        \draw[black, thick] (2,-2) -- (3,-2);
        \draw[black, thick] (4,-2) -- (3.5,-3);
        \draw[black, thick] (4,-2) -- (4.5,-3);
        \draw[black, thick] (4.5,-3) -- (3.5,-3);
        \draw[black, thick] (4,-2) -- (3.5,-1);
        \draw[black, thick] (3,-2) -- (2.5,-3);
        \draw[black, thick] (4,0) -- (5.5,-3);
        \draw[black, thick] (5,-4) -- (5.5,-3);
        \draw[black, thick] (5,-4) -- (2,-4);
        \draw[black, thick] (5,0) -- (7.5,-5);
        \draw[black, thick] (6,0) -- (8,-4);
        \draw[black, thick] (7,0) -- (8.5,-3);
        \draw[black, thick] (8,0) -- (9,-2);
        \draw[black, thick] (5.5,-9) -- (4.5,-9);
        \draw[black, thick] (6,-8) -- (4,-8);
        \draw[black, thick] (6.5,-7) -- (3.5,-7);
        \draw[black, thick] (7,-6) -- (3,-6);
        \draw[black, thick] (2.5,-5) -- (5.5,-5);
        \draw[black, thick] (6.5,-3) -- (5.5,-5);
        \draw[black, thick] (6,-4) -- (7,-4);
        \draw[black, thick] (6,-6) -- (6.5,-5);
        \draw[black, thick] (7,-6) -- (6.5,-5);
        
        \draw (1,-0.6) node[anchor=center, scale = 0.7] {0};
        \draw (2,-0.6) node[anchor=center, scale = 0.7] {0};
        \draw (3,-0.6) node[anchor=center, scale = 0.7] {1};
        \draw (4,-0.6) node[anchor=center, scale = 0.7] {0};
        \draw (5,-0.6) node[anchor=center, scale = 0.7] {0};
        \draw (6,-0.6) node[anchor=center, scale = 0.7] {0};
        \draw (7,-0.6) node[anchor=center, scale = 0.7] {0};
        \draw (8,-0.6) node[anchor=center, scale = 0.7] {0};
        \draw (9,-0.6) node[anchor=center, scale = 0.7] {0};
        \draw (1.5,-1.6) node[anchor=center, scale = 0.7] {0};
        \draw (2.5,-1.6) node[anchor=center, scale = 0.7] {1};
        \draw (3.5,-1.6) node[anchor=center, scale = 0.7] {0};
        \draw (4.5,-1.6) node[anchor=center, scale = 0.7] {0};
        \draw (5.5,-1.6) node[anchor=center, scale = 0.7] {0};
        \draw (6.5,-1.6) node[anchor=center, scale = 0.7] {0};
        \draw (7.5,-1.6) node[anchor=center, scale = 0.7] {0};
        \draw (8.5,-1.6) node[anchor=center, scale = 0.7] {0};
        \draw (2,-2.6) node[anchor=center, scale = 0.7] {0};
        \draw (3,-2.6) node[anchor=center, scale = 0.7] {0};
        \draw (4,-2.6) node[anchor=center, scale = 0.7] {1};
        \draw (5,-2.6) node[anchor=center, scale = 0.7] {0};
        \draw (6,-2.6) node[anchor=center, scale = 0.7] {0};
        \draw (7,-2.6) node[anchor=center, scale = 0.7] {0};
        \draw (8,-2.6) node[anchor=center, scale = 0.7] {0};
        \draw (2.5,-3.6) node[anchor=center, scale = 0.7] {0};
        \draw (3.5,-3.6) node[anchor=center, scale = 0.7] {0};
        \draw (4.5,-3.6) node[anchor=center, scale = 0.7] {0};
        \draw (5.5,-3.6) node[anchor=center, scale = 0.7] {0};
        \draw (6.5,-3.6) node[anchor=center, scale = 0.7] {1};
        \draw (7.5,-3.6) node[anchor=center, scale = 0.7] {0};
        \draw (3,-4.6) node[anchor=center, scale = 0.7] {0};
        \draw (4,-4.6) node[anchor=center, scale = 0.7] {0};
        \draw (5,-4.6) node[anchor=center, scale = 0.7] {0};
        \draw (6,-4.6) node[anchor=center, scale = 0.7] {0};
        \draw (7,-4.6) node[anchor=center, scale = 0.7] {0};
        \draw (3.5,-5.6) node[anchor=center, scale = 0.7] {0};
        \draw (4.5,-5.6) node[anchor=center, scale = 0.7] {0};
        \draw (5.5,-5.6) node[anchor=center, scale = 0.7] {0};
        \draw (6.5,-5.6) node[anchor=center, scale = 0.7] {1};
        \draw (4,-6.6) node[anchor=center, scale = 0.7] {0};
        \draw (5,-6.6) node[anchor=center, scale = 0.7] {0};
        \draw (6,-6.6) node[anchor=center, scale = 0.7] {0};
        \draw (4.5,-7.6) node[anchor=center, scale = 0.7] {0};
        \draw (5.5,-7.6) node[anchor=center, scale = 0.7] {0};
        \draw (5,-8.6) node[anchor=center, scale = 0.7] {0};
        
        \draw[black, dotted] (1.5,-0.4) -- (1,-1.4);
        \draw[black, dotted] (2.5,-0.4) -- (1.5,-2.4);
        \draw[black, dotted] (3.5,-0.4) -- (5,-3.4);
        \draw[black, dotted] (4.5,-0.4) -- (6,-3.4);
        \draw[black, dotted] (5.5,-0.4) -- (7.5,-4.4);
        \draw[black, dotted] (6.5,-0.4) -- (8,-3.4);
        \draw[black, dotted] (7.5,-0.4) -- (8.5,-2.4);
        \draw[black, dotted] (8.5,-0.4) -- (9,-1.4);
        \draw[black, dotted] (3.5,-2.4) -- (3,-1.4);
        \draw[black, dotted] (2.5,-2.4) -- (2,-3.4);
        \draw[black, dotted] (5,-3.4) -- (2,-3.4);
        \draw[black, dotted] (2.5,-4.4) -- (5.5,-4.4);
        \draw[black, dotted] (7,-5.4) -- (6.5,-4.4);
        \draw[black, dotted] (6,-5.4) -- (6.5,-4.4);
        \draw[black, dotted] (6,-5.4) -- (3,-5.4);
        \draw[black, dotted] (3.5,-6.4) -- (6.5,-6.4);
        \draw[black, dotted] (4,-7.4) -- (6,-7.4);
        \draw[black, dotted] (4.5,-8.4) -- (5.5,-8.4);
        \draw[black, dotted] (6,-3.4) -- (5.5,-4.4);
        \draw[black, dotted] (3,-3.4) -- (3.5,-2.4);
        
        \filldraw[red] (0.5,-0.4) circle (2pt);
        \filldraw[red] (1.5,-0.4) circle (2pt);
        \filldraw[red] (2.5,-0.4) circle (2pt);
        \filldraw[red] (3.5,-0.4) circle (2pt);
        \filldraw[red] (4.5,-0.4) circle (2pt);
        \filldraw[red] (5.5,-0.4) circle (2pt);
        \filldraw[red] (6.5,-0.4) circle (2pt);
        \filldraw[red] (7.5,-0.4) circle (2pt);
        \filldraw[red] (8.5,-0.4) circle (2pt);
        \filldraw[red] (9.5,-0.4) circle (2pt);
        \filldraw[red] (1,-1.4) circle (2pt);
        \filldraw[red] (2,-1.4) circle (2pt);
        \filldraw[red] (3,-1.4) circle (2pt);
        \filldraw[red] (4,-1.4) circle (2pt);
        \filldraw[red] (5,-1.4) circle (2pt);
        \filldraw[red] (6,-1.4) circle (2pt);
        \filldraw[red] (7,-1.4) circle (2pt);
        \filldraw[red] (8,-1.4) circle (2pt);
        \filldraw[red] (9,-1.4) circle (2pt);
        \filldraw[red] (1.5,-2.4) circle (2pt);
        \filldraw[red] (2.5,-2.4) circle (2pt);
        \filldraw[red] (3.5,-2.4) circle (2pt);
        \filldraw[red] (4.5,-2.4) circle (2pt);
        \filldraw[red] (5.5,-2.4) circle (2pt);
        \filldraw[red] (6.5,-2.4) circle (2pt);
        \filldraw[red] (7.5,-2.4) circle (2pt);
        \filldraw[red] (8.5,-2.4) circle (2pt);
        \filldraw[red] (2,-3.4) circle (2pt);
        \filldraw[red] (3,-3.4) circle (2pt);
        \filldraw[red] (4,-3.4) circle (2pt);
        \filldraw[red] (5,-3.4) circle (2pt);
        \filldraw[red] (6,-3.4) circle (2pt);
        \filldraw[red] (7,-3.4) circle (2pt);
        \filldraw[red] (8,-3.4) circle (2pt);
        \filldraw[red] (2.5,-4.4) circle (2pt);
        \filldraw[red] (3.5,-4.4) circle (2pt);
        \filldraw[red] (4.5,-4.4) circle (2pt);
        \filldraw[red] (5.5,-4.4) circle (2pt);
        \filldraw[red] (6.5,-4.4) circle (2pt);
        \filldraw[red] (7.5,-4.4) circle (2pt);
        \filldraw[red] (3,-5.4) circle (2pt);
        \filldraw[red] (4,-5.4) circle (2pt);
        \filldraw[red] (5,-5.4) circle (2pt);
        \filldraw[red] (6,-5.4) circle (2pt);
        \filldraw[red] (7,-5.4) circle (2pt);
        \filldraw[red] (3.5,-6.4) circle (2pt);
        \filldraw[red] (4.5,-6.4) circle (2pt);
        \filldraw[red] (5.5,-6.4) circle (2pt);
        \filldraw[red] (6.5,-6.4) circle (2pt);
        \filldraw[red] (4,-7.4) circle (2pt);
        \filldraw[red] (5,-7.4) circle (2pt);
        \filldraw[red] (6,-7.4) circle (2pt);
        \filldraw[red] (4.5,-8.4) circle (2pt);
        \filldraw[red] (5.5,-8.4) circle (2pt);
        \filldraw[red] (5,-9.4) circle (2pt);
        
    \end{tikzpicture}
    }
            \caption{Connecting red vertices}
        \end{subfigure}
     \hfill
        \begin{subfigure}[b]{7.5cm}
            \centering
    
    \resizebox{\textwidth}{!}{
    \begin{tikzpicture}
    
        \draw (1,-0.6) node[anchor=center, scale = 0.7] {0};
        \draw (2,-0.6) node[anchor=center, scale = 0.7] {0};
        \draw (3,-0.6) node[anchor=center, scale = 0.7] {1};
        \draw (4,-0.6) node[anchor=center, scale = 0.7] {0};
        \draw (5,-0.6) node[anchor=center, scale = 0.7] {0};
        \draw (6,-0.6) node[anchor=center, scale = 0.7] {0};
        \draw (7,-0.6) node[anchor=center, scale = 0.7] {0};
        \draw (8,-0.6) node[anchor=center, scale = 0.7] {0};
        \draw (9,-0.6) node[anchor=center, scale = 0.7] {0};
        \draw (1.5,-1.6) node[anchor=center, scale = 0.7] {0};
        \draw (2.5,-1.6) node[anchor=center, scale = 0.7] {1};
        \draw (3.5,-1.6) node[anchor=center, scale = 0.7] {0};
        \draw (4.5,-1.6) node[anchor=center, scale = 0.7] {0};
        \draw (5.5,-1.6) node[anchor=center, scale = 0.7] {0};
        \draw (6.5,-1.6) node[anchor=center, scale = 0.7] {0};
        \draw (7.5,-1.6) node[anchor=center, scale = 0.7] {0};
        \draw (8.5,-1.6) node[anchor=center, scale = 0.7] {0};
        \draw (2,-2.6) node[anchor=center, scale = 0.7] {0};
        \draw (3,-2.6) node[anchor=center, scale = 0.7] {0};
        \draw (4,-2.6) node[anchor=center, scale = 0.7] {1};
        \draw (5,-2.6) node[anchor=center, scale = 0.7] {0};
        \draw (6,-2.6) node[anchor=center, scale = 0.7] {0};
        \draw (7,-2.6) node[anchor=center, scale = 0.7] {0};
        \draw (8,-2.6) node[anchor=center, scale = 0.7] {0};
        \draw (2.5,-3.6) node[anchor=center, scale = 0.7] {0};
        \draw (3.5,-3.6) node[anchor=center, scale = 0.7] {0};
        \draw (4.5,-3.6) node[anchor=center, scale = 0.7] {0};
        \draw (5.5,-3.6) node[anchor=center, scale = 0.7] {0};
        \draw (6.5,-3.6) node[anchor=center, scale = 0.7] {1};
        \draw (7.5,-3.6) node[anchor=center, scale = 0.7] {0};
        \draw (3,-4.6) node[anchor=center, scale = 0.7] {0};
        \draw (4,-4.6) node[anchor=center, scale = 0.7] {0};
        \draw (5,-4.6) node[anchor=center, scale = 0.7] {0};
        \draw (6,-4.6) node[anchor=center, scale = 0.7] {0};
        \draw (7,-4.6) node[anchor=center, scale = 0.7] {0};
        \draw (3.5,-5.6) node[anchor=center, scale = 0.7] {0};
        \draw (4.5,-5.6) node[anchor=center, scale = 0.7] {0};
        \draw (5.5,-5.6) node[anchor=center, scale = 0.7] {0};
        \draw (6.5,-5.6) node[anchor=center, scale = 0.7] {1};
        \draw (4,-6.6) node[anchor=center, scale = 0.7] {0};
        \draw (5,-6.6) node[anchor=center, scale = 0.7] {0};
        \draw (6,-6.6) node[anchor=center, scale = 0.7] {0};
        \draw (4.5,-7.6) node[anchor=center, scale = 0.7] {0};
        \draw (5.5,-7.6) node[anchor=center, scale = 0.7] {0};
        \draw (5,-8.6) node[anchor=center, scale = 0.7] {0};
        
        \draw[black] (1.5,-0.2) -- (1,-1.2);
        \draw[black] (2.5,-0.2) -- (1.5,-2.2);
        \draw[black] (3.5,-0.2) -- (5,-3.2);
        \draw[black] (4.5,-0.2) -- (6,-3.2);
        \draw[black] (5.5,-0.2) -- (7.5,-4.2);
        \draw[black] (6.5,-0.2) -- (8,-3.2);
        \draw[black] (7.5,-0.2) -- (8.5,-2.2);
        \draw[black] (8.5,-0.2) -- (9,-1.2);
        \draw[black] (3.5,-2.2) -- (3,-1.2);
        \draw[black] (2.5,-2.2) -- (2,-3.2);
        \draw[black] (5,-3.2) -- (2,-3.2);
        \draw[black] (2.5,-4.2) -- (5.5,-4.2);
        \draw[black] (7,-5.2) -- (6.5,-4.2);
        \draw[black] (6,-5.2) -- (6.5,-4.2);
        \draw[black] (6,-5.2) -- (3,-5.2);
        \draw[black] (3.5,-6.2) -- (6.5,-6.2);
        \draw[black] (4,-7.2) -- (6,-7.2);
        \draw[black] (4.5,-8.2) -- (5.5,-8.2);
        \draw[black] (6,-3.2) -- (5.5,-4.2);
        \draw[black] (3,-3.2) -- (3.5,-2.2);
        
        \filldraw[red] (0.5,-0.2) circle (2pt);
        \filldraw[red] (1.5,-0.2) circle (2pt);
        \filldraw[red] (2.5,-0.2) circle (2pt);
        \filldraw[red] (3.5,-0.2) circle (2pt);
        \filldraw[red] (4.5,-0.2) circle (2pt);
        \filldraw[red] (5.5,-0.2) circle (2pt);
        \filldraw[red] (6.5,-0.2) circle (2pt);
        \filldraw[red] (7.5,-0.2) circle (2pt);
        \filldraw[red] (8.5,-0.2) circle (2pt);
        \filldraw[red] (9.5,-0.2) circle (2pt);
        \filldraw[red] (1,-1.2) circle (2pt);
        \filldraw[red] (2,-1.2) circle (2pt);
        \filldraw[red] (3,-1.2) circle (2pt);
        \filldraw[red] (4,-1.2) circle (2pt);
        \filldraw[red] (5,-1.2) circle (2pt);
        \filldraw[red] (6,-1.2) circle (2pt);
        \filldraw[red] (7,-1.2) circle (2pt);
        \filldraw[red] (8,-1.2) circle (2pt);
        \filldraw[red] (9,-1.2) circle (2pt);
        \filldraw[red] (1.5,-2.2) circle (2pt);
        \filldraw[red] (2.5,-2.2) circle (2pt);
        \filldraw[red] (3.5,-2.2) circle (2pt);
        \filldraw[red] (4.5,-2.2) circle (2pt);
        \filldraw[red] (5.5,-2.2) circle (2pt);
        \filldraw[red] (6.5,-2.2) circle (2pt);
        \filldraw[red] (7.5,-2.2) circle (2pt);
        \filldraw[red] (8.5,-2.2) circle (2pt);
        \filldraw[red] (2,-3.2) circle (2pt);
        \filldraw[red] (3,-3.2) circle (2pt);
        \filldraw[red] (4,-3.2) circle (2pt);
        \filldraw[red] (5,-3.2) circle (2pt);
        \filldraw[red] (6,-3.2) circle (2pt);
        \filldraw[red] (7,-3.2) circle (2pt);
        \filldraw[red] (8,-3.2) circle (2pt);
        \filldraw[red] (2.5,-4.2) circle (2pt);
        \filldraw[red] (3.5,-4.2) circle (2pt);
        \filldraw[red] (4.5,-4.2) circle (2pt);
        \filldraw[red] (5.5,-4.2) circle (2pt);
        \filldraw[red] (6.5,-4.2) circle (2pt);
        \filldraw[red] (7.5,-4.2) circle (2pt);
        \filldraw[red] (3,-5.2) circle (2pt);
        \filldraw[red] (4,-5.2) circle (2pt);
        \filldraw[red] (5,-5.2) circle (2pt);
        \filldraw[red] (6,-5.2) circle (2pt);
        \filldraw[red] (7,-5.2) circle (2pt);
        \filldraw[red] (3.5,-6.2) circle (2pt);
        \filldraw[red] (4.5,-6.2) circle (2pt);
        \filldraw[red] (5.5,-6.2) circle (2pt);
        \filldraw[red] (6.5,-6.2) circle (2pt);
        \filldraw[red] (4,-7.2) circle (2pt);
        \filldraw[red] (5,-7.2) circle (2pt);
        \filldraw[red] (6,-7.2) circle (2pt);
        \filldraw[red] (4.5,-8.2) circle (2pt);
        \filldraw[red] (5.5,-8.2) circle (2pt);
        \filldraw[red] (5,-9.2) circle (2pt);
        
    \end{tikzpicture}
    }
            \caption{Resulting grove}
        \end{subfigure}
        
        \caption{Final grove}
        \label{Figure 6}
        
    \end{figure}
	
	\justify
	\textbf{Remark:} An interesting point to note from our proof is that in every permutation triangle of size $n$, there is a permutation triangle of size $n-1$. This is also true for alternating sign triangles. The proof for this remark will not be discussed here since it use the same argument as the proof above. Though this remark will not be discussed further in this paper, this may lead to some potential recursive relations among alternating sign triangles.
	
	\section{Alternating sign triangles}
	
	It is still an open problem to give a simple direct characterization of alternating sign triangles, without going through groves. Nevertheless, in this section, we will discuss some properties of alternating sign triangles. Note that due to symmetry, all properties discussed below are also true for the left-columns and right-columns.
	
	\begin{property}
	    The sum of the entries in any isosceles trapezoid of height $i$ whose top row lies on the top row of the alternating sign triangle (in case the bottom row contains only $1$ entry, this become a downward triangle) is at most $i-v$ where $v$ is the number of $-1$ on the boundary of the trapezoid.
	\end{property}
	
	\begin{proof}
	  The proof of this property use the same argument as property $3$ of permutation triangles. Consider the corresponding grove. Assume that there are $k$ vertices on the first row, then there are $\frac{(k+k-i)(i+1)}{2}$ vertices in the trapezoid, $2k+i-2$ of which lied on the boundary. Each top vertex except the two left-most and right-most has to be connected with at least one other boundary vertex, and each $-1$ on the boundary connect two boundary vertices. Therefore, the boundary vertices are divided into at most $2k+i-2-(k-2)-v=k+i-v$ components. Therefore, there are at least $\frac{(k+k-i)(i+1)}{2}-(k+i-v)=ki-\frac{i^2+i}{2}-i+v$ edges. However, there are $\frac{(k-1+k-i)i}{2}=ki-\frac{i^2+i}{2}$ triangles. Hence, the sum of all entries is at most $i-v$.
	\end{proof}
    
    \begin{figure}[h!]
        \centering
        \begin{tikzpicture}
        
        \draw[blue, dashed] (1,0) -- (7,0);
        \draw[blue, dashed] (5.5,-3) -- (7,0);
        \draw[blue, dashed] (5.5,-3) -- (2.5,-3);
        \draw[blue, dashed] (1,0) -- (2.5,-3);
        \draw[fill=cyan, nearly transparent]  (1,0) -- (7,0) -- (5.5,-3) -- (2.5,-3) -- cycle;
        
        \filldraw[black] (0,0) circle (2pt);
        \filldraw[black] (1,0) circle (2pt);
        \filldraw[black] (2,0) circle (2pt);
        \filldraw[black] (3,0) circle (2pt);
        \filldraw[black] (4,0) circle (2pt);
        \filldraw[black] (5,0) circle (2pt);
        \filldraw[black] (6,0) circle (2pt);
        \filldraw[black] (7,0) circle (2pt);
        \filldraw[black] (8,0) circle (2pt);
        \filldraw[black] (9,0) circle (2pt);
        \filldraw[black] (0.5,-1) circle (2pt);
        \filldraw[black] (1.5,-1) circle (2pt);
        \filldraw[black] (2.5,-1) circle (2pt);
        \filldraw[black] (3.5,-1) circle (2pt);
        \filldraw[black] (4.5,-1) circle (2pt);
        \filldraw[black] (5.5,-1) circle (2pt);
        \filldraw[black] (6.5,-1) circle (2pt);
        \filldraw[black] (7.5,-1) circle (2pt);
        \filldraw[black] (8.5,-1) circle (2pt);
        \filldraw[black] (1,-2) circle (2pt);
        \filldraw[black] (2,-2) circle (2pt);
        \filldraw[black] (3,-2) circle (2pt);
        \filldraw[black] (4,-2) circle (2pt);
        \filldraw[black] (5,-2) circle (2pt);
        \filldraw[black] (6,-2) circle (2pt);
        \filldraw[black] (7,-2) circle (2pt);
        \filldraw[black] (8,-2) circle (2pt);
        \filldraw[black] (1.5,-3) circle (2pt);
        \filldraw[black] (2.5,-3) circle (2pt);
        \filldraw[black] (3.5,-3) circle (2pt);
        \filldraw[black] (4.5,-3) circle (2pt);
        \filldraw[black] (5.5,-3) circle (2pt);
        \filldraw[black] (6.5,-3) circle (2pt);
        \filldraw[black] (7.5,-3) circle (2pt);
        
        \draw[black, thick] (2,0) -- (1.5,-1);
        \draw[black, thick] (2.5,-1) -- (1.5,-1);
        \draw[black, thick] (2.5,-1) -- (2,-2);
        \draw[black, thick] (3,0) -- (3.5,-1);
        \draw[black, thick] (3,-2) -- (3.5,-1);
        \draw[black, thick] (3,-2) -- (3.5,-3);
        \draw[black, thick] (4,0) -- (5.5,-3);
        \draw[black, thick] (4,-2) -- (5,-2);
        \draw[black, thick] (5,0) -- (6,-2);
        \draw[black, thick] (6,0) -- (6.5,-1);
        
        \draw (1.5,-0.4) node[anchor=center, scale = 0.7] {0};
        \draw (2.5,-0.4) node[anchor=center, scale = 0.7] {1};
        \draw (3.5,-0.4) node[anchor=center, scale = 0.7] {0};
        \draw (4.5,-0.4) node[anchor=center, scale = 0.7] {0};
        \draw (5.5,-0.4) node[anchor=center, scale = 0.7] {0};
        \draw (6.5,-0.4) node[anchor=center, scale = 0.7] {0};
        \draw (2,-1.4) node[anchor=center, scale = 0.7] {-1};
        \draw (3,-1.4) node[anchor=center, scale = 0.7] {0};
        \draw (4,-1.4) node[anchor=center, scale = 0.7] {1};
        \draw (5,-1.4) node[anchor=center, scale = 0.7] {0};
        \draw (6,-1.4) node[anchor=center, scale = 0.7] {0};
        \draw (2.5,-2.4) node[anchor=center, scale = 0.7] {1};
        \draw (3.5,-2.4) node[anchor=center, scale = 0.7] {0};
        \draw (4.5,-2.4) node[anchor=center, scale = 0.7] {0};
        \draw (5.5,-2.4) node[anchor=center, scale = 0.7] {0};
        \end{tikzpicture}
        \caption{Example with $i=3,v=1$}
        \label{Figure 7}
    \end{figure}
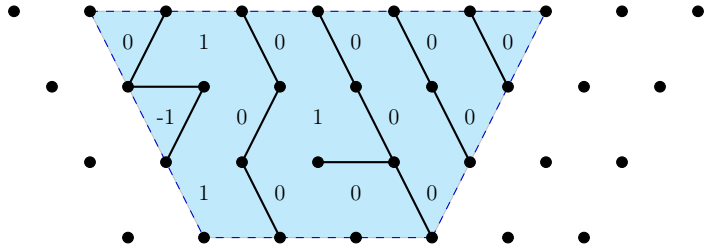
	
	\justify
	\textbf{Remark:} This property is a generalization of property $3$ of permutation triangles.
	
	\begin{property}
	    The sum of the entries in an upward triangle of size $i$ is at most $i$.
	\end{property}
	
	    This property can be proven using the same argument as property $6$ of permutation triangles.
	
	\begin{property}
	    The $i$th row has at most $i$ $1$-s and $i-1$ $-1$-s.
	\end{property}
	
	\begin{proof}
	    Consider a sub-graph whose vertices are those in the first $i+1$ rows of the grove and edges are those connecting such vertices, excluding edges connecting the vertices in row $i+1$. The connectivity conditions state that the vertices $(-n+2i,0),(-n+2i+2,0),...,(n-2i,0)$ have to be connected with some vertices in row $i+1$ and below. Therefore, at least $n+1-2i$ vertices in row $i+1$ have to be connected with a vertex in row $1$, which means there are at least $n+1-2i$ $0$-s in row $i$. Hence, there are at most $i$ $1$-s in row $i$.
	    
	    On the other hand, at least $n+1-2i+2$ vertices in row $i$ have to be connected with a vertex in row $1$. This means that the vertices in row $i$ are divided into at least $n+1-2i+2$ components. Since each $-1$ in row $i$ reduces the number of components by $1$, there are at most $i-1$ $-1$-s in row $i$.
	\end{proof}
	
	\begin{property}
	    In the union of any set of downward triangles whose first row contains entries in the first row of the alternating sign triangles, the sum of the entries is non-negative. 
	\end{property}
	
	\begin{proof}
	    Let the number of vertices in the union be $k$, and $A$ be the set of the vertices in the union that are also in the first row. The vertices in $A$ needs to be in different components; hence, in the union, there are at least $\lvert A\rvert$ components. Therefore, there are at most $k-\lvert A\rvert$ edges in the union. On the other hand, for every vertex in the union that are not in $A$, the downward triangle above it is also in the union; conversely, for any downward triangle in the union, its bottom vertex is also in the union and is also not in $A$. Therefore, the number of downward triangles is exactly $k-\lvert A\rvert$. Hence, the sum of the entries is non-negative.
	\end{proof}
    
    \begin{figure}[h!]
        \centering
        \begin{tikzpicture}

        \filldraw[black] (0,0) circle (2pt);
        \filldraw[black] (1,0) circle (2pt);
        \filldraw[black] (2,0) circle (2pt);
        \filldraw[black] (3,0) circle (2pt);
        \filldraw[black] (4,0) circle (2pt);
        \filldraw[black] (5,0) circle (2pt);
        \filldraw[black] (6,0) circle (2pt);
        \filldraw[black] (7,0) circle (2pt);
        \filldraw[black] (8,0) circle (2pt);
        \filldraw[black] (0.5,-1) circle (2pt);
        \filldraw[black] (1.5,-1) circle (2pt);
        \filldraw[black] (2.5,-1) circle (2pt);
        \filldraw[black] (3.5,-1) circle (2pt);
        \filldraw[black] (4.5,-1) circle (2pt);
        \filldraw[black] (5.5,-1) circle (2pt);
        \filldraw[black] (6.5,-1) circle (2pt);
        \filldraw[black] (7.5,-1) circle (2pt);
        \filldraw[black] (1,-2) circle (2pt);
        \filldraw[black] (2,-2) circle (2pt);
        \filldraw[black] (3,-2) circle (2pt);
        \filldraw[black] (4,-2) circle (2pt);
        \filldraw[black] (5,-2) circle (2pt);
        \filldraw[black] (6,-2) circle (2pt);
        \filldraw[black] (7,-2) circle (2pt);
        \filldraw[black] (1.5,-3) circle (2pt);
        \filldraw[black] (2.5,-3) circle (2pt);
        \filldraw[black] (3.5,-3) circle (2pt);
        \filldraw[black] (4.5,-3) circle (2pt);
        \filldraw[black] (5.5,-3) circle (2pt);
        \filldraw[black] (6.5,-3) circle (2pt);
        \filldraw[black] (2,-4) circle (2pt);
        \filldraw[black] (3,-4) circle (2pt);
        \filldraw[black] (4,-4) circle (2pt);
        \filldraw[black] (5,-4) circle (2pt);
        \filldraw[black] (6,-4) circle (2pt);
        \filldraw[black] (2.5,-5) circle (2pt);
        \filldraw[black] (3.5,-5) circle (2pt);
        \filldraw[black] (4.5,-5) circle (2pt);
        \filldraw[black] (5.5,-5) circle (2pt);
        \filldraw[black] (3,-6) circle (2pt);
        \filldraw[black] (4,-6) circle (2pt);
        \filldraw[black] (5,-6) circle (2pt);
        \filldraw[black] (3.5,-7) circle (2pt);
        \filldraw[black] (4.5,-7) circle (2pt);
        \filldraw[black] (4,-8) circle (2pt);
        
        \draw[black, thick] (1,0) -- (0.5,-1);
        \draw[black, thick] (2,0) -- (1,-2);
        \draw[black, thick] (3,0) -- (1.5,-3);
        \draw[black, thick] (2.5,-3) -- (2,-2);
        \draw[black, thick] (2.5,-3) -- (3,-2);
        \draw[black, thick] (4,0) -- (5.5,-3);
        \draw[black, thick] (5,0) -- (6.5,-3);
        \draw[black, thick] (6,0) -- (7,-2);
        \draw[black, thick] (7,0) -- (7.5,-1);
        \draw[black, thick] (4.5,-3) -- (5.5,-3);
        \draw[black, thick] (3.5,-1) -- (5,-4);
        \draw[black, thick] (6,-4) -- (5,-4);
        \draw[black, thick] (2,-4) -- (3,-4);
        \draw[black, thick] (4,-2) -- (3,-4);
        \draw[black, thick] (2.5,-5) -- (3.5,-5);
        \draw[black, thick] (4,-4) -- (3.5,-5);
        \draw[black, thick] (4,-4) -- (4.5,-5);
        \draw[black, thick] (5.5,-5) -- (4.5,-5);
        \draw[black, thick] (3,-6) -- (5,-6);
        \draw[black, thick] (3.5,-7) -- (4.5,-7);
        
        \draw (0.5,-0.4) node[anchor=center, scale = 0.7] {0};
        \draw (1.5,-0.4) node[anchor=center, scale = 0.7] {0};
        \draw (2.5,-0.4) node[anchor=center, scale = 0.7] {0};
        \draw (3.5,-0.4) node[anchor=center, scale = 0.7] {1};
        \draw (4.5,-0.4) node[anchor=center, scale = 0.7] {0};
        \draw (5.5,-0.4) node[anchor=center, scale = 0.7] {0};
        \draw (6.5,-0.4) node[anchor=center, scale = 0.7] {0};
        \draw (7.5,-0.4) node[anchor=center, scale = 0.7] {0};
        \draw (1,-1.4) node[anchor=center, scale = 0.7] {0};
        \draw (2,-1.4) node[anchor=center, scale = 0.7] {0};
        \draw (3,-1.4) node[anchor=center, scale = 0.7] {1};
        \draw (4,-1.4) node[anchor=center, scale = 0.7] {0};
        \draw (5,-1.4) node[anchor=center, scale = 0.7] {0};
        \draw (6,-1.4) node[anchor=center, scale = 0.7] {0};
        \draw (7,-1.4) node[anchor=center, scale = 0.7] {0};
        \draw (1.5,-2.4) node[anchor=center, scale = 0.7] {0};
        \draw (2.5,-2.4) node[anchor=center, scale = 0.7] {-1};
        \draw (3.5,-2.4) node[anchor=center, scale = 0.7] {0};
        \draw (4.5,-2.4) node[anchor=center, scale = 0.7] {0};
        \draw (5.5,-2.4) node[anchor=center, scale = 0.7] {0};
        \draw (6.5,-2.4) node[anchor=center, scale = 0.7] {0};
        \draw (2,-3.4) node[anchor=center, scale = 0.7] {1};
        \draw (3,-3.4) node[anchor=center, scale = 0.7] {0};
        \draw (4,-3.4) node[anchor=center, scale = 0.7] {1};
        \draw (5,-3.4) node[anchor=center, scale = 0.7] {-1};
        \draw (6,-3.4) node[anchor=center, scale = 0.7] {1};
        \draw (2.5,-4.4) node[anchor=center, scale = 0.7] {0};
        \draw (3.5,-4.4) node[anchor=center, scale = 0.7] {0};
        \draw (4.5,-4.4) node[anchor=center, scale = 0.7] {0};
        \draw (5.5,-4.4) node[anchor=center, scale = 0.7] {0};
        \draw (3,-5.4) node[anchor=center, scale = 0.7] {0};
        \draw (4,-5.4) node[anchor=center, scale = 0.7] {1};
        \draw (5,-5.4) node[anchor=center, scale = 0.7] {0};
        \draw (3.5,-6.4) node[anchor=center, scale = 0.7] {0};
        \draw (4.5,-6.4) node[anchor=center, scale = 0.7] {0};
        \draw (4,-7.4) node[anchor=center, scale = 0.7] {0};

        \draw[blue, dashed] (1,0) -- (2.5,-3);
        \draw[blue, dashed] (3.5,-1) -- (2.5,-3);
        \draw[blue, dashed] (3.5,-1) -- (5,-4);
        \draw[blue, dashed] (7,0) -- (5,-4);
        \draw[blue, dashed] (7,0) -- (1,0);
        \draw[fill=cyan, nearly transparent]  (1,0) -- (7,0) -- (5,-4) -- (3.5,-1) -- (2.5,-3) -- cycle;        
        
        \end{tikzpicture}
        \caption{Property 4}
        \label{Figure 8}
    \end{figure}
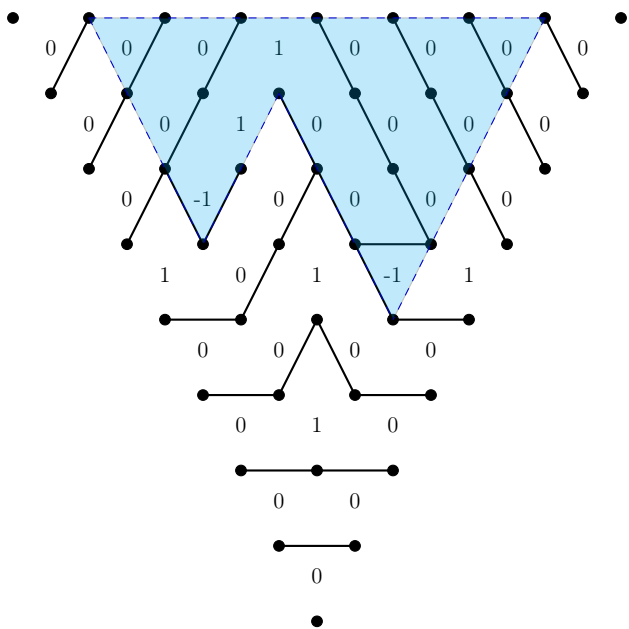
	
	\begin{corollary}
	    For any $-1$-entry, the downward triangle above it has to contain at least $1$ $1$-entry. Similarly, the downward triangles to the left and right of it have to contain at least $1$ $1$-entry.
	\end{corollary}
    
    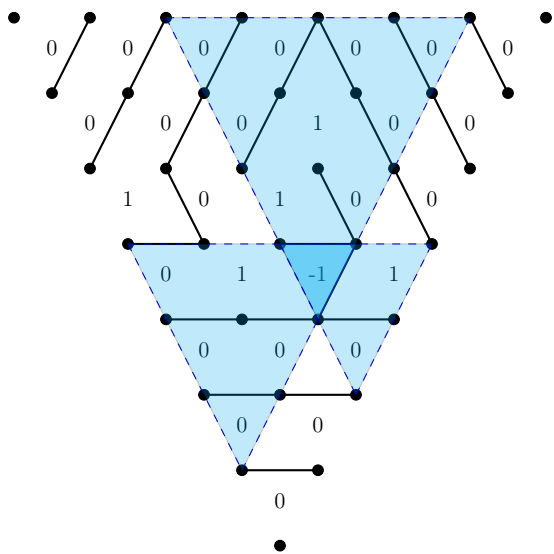
\begin{figure}[h!]
        \centering
        \begin{tikzpicture}

        \filldraw[black] (0,0) circle (2pt);
        \filldraw[black] (1,0) circle (2pt);
        \filldraw[black] (2,0) circle (2pt);
        \filldraw[black] (3,0) circle (2pt);
        \filldraw[black] (4,0) circle (2pt);
        \filldraw[black] (5,0) circle (2pt);
        \filldraw[black] (6,0) circle (2pt);
        \filldraw[black] (7,0) circle (2pt);
        \filldraw[black] (0.5,-1) circle (2pt);
        \filldraw[black] (1.5,-1) circle (2pt);
        \filldraw[black] (2.5,-1) circle (2pt);
        \filldraw[black] (3.5,-1) circle (2pt);
        \filldraw[black] (4.5,-1) circle (2pt);
        \filldraw[black] (5.5,-1) circle (2pt);
        \filldraw[black] (6.5,-1) circle (2pt);
        \filldraw[black] (1,-2) circle (2pt);
        \filldraw[black] (2,-2) circle (2pt);
        \filldraw[black] (3,-2) circle (2pt);
        \filldraw[black] (4,-2) circle (2pt);
        \filldraw[black] (5,-2) circle (2pt);
        \filldraw[black] (6,-2) circle (2pt);
        \filldraw[black] (1.5,-3) circle (2pt);
        \filldraw[black] (2.5,-3) circle (2pt);
        \filldraw[black] (3.5,-3) circle (2pt);
        \filldraw[black] (4.5,-3) circle (2pt);
        \filldraw[black] (5.5,-3) circle (2pt);
        \filldraw[black] (2,-4) circle (2pt);
        \filldraw[black] (3,-4) circle (2pt);
        \filldraw[black] (4,-4) circle (2pt);
        \filldraw[black] (5,-4) circle (2pt);
        \filldraw[black] (2.5,-5) circle (2pt);
        \filldraw[black] (3.5,-5) circle (2pt);
        \filldraw[black] (4.5,-5) circle (2pt);
        \filldraw[black] (3,-6) circle (2pt);
        \filldraw[black] (4,-6) circle (2pt);
        \filldraw[black] (3.5,-7) circle (2pt);
        
        \draw[black, thick] (1,0) -- (0.5,-1);
        \draw[black, thick] (2,0) -- (1,-2);
        \draw[black, thick] (3,0) -- (2,-2);
        \draw[black, thick] (2.5,-3) -- (2,-2);
        \draw[black, thick] (2.5,-3) -- (1.5,-3);
        \draw[black, thick] (4,0) -- (3,-2);
        \draw[black, thick] (4,0) -- (5.5,-3);
        \draw[black, thick] (5,0) -- (6,-2);
        \draw[black, thick] (6,0) -- (6.5,-1);
        \draw[black, thick] (4.5,-3) -- (4,-2);
        \draw[black, thick] (4.5,-3) -- (3.5,-3);
        \draw[black, thick] (4.5,-3) -- (4,-4);
        \draw[black, thick] (2,-4) -- (5,-4);
        \draw[black, thick] (2.5,-5) -- (4.5,-5);
        \draw[black, thick] (3,-6) -- (4,-6);
        
        \draw (0.5,-0.4) node[anchor=center, scale = 0.7] {0};
        \draw (1.5,-0.4) node[anchor=center, scale = 0.7] {0};
        \draw (2.5,-0.4) node[anchor=center, scale = 0.7] {0};
        \draw (3.5,-0.4) node[anchor=center, scale = 0.7] {0};
        \draw (4.5,-0.4) node[anchor=center, scale = 0.7] {0};
        \draw (5.5,-0.4) node[anchor=center, scale = 0.7] {0};
        \draw (6.5,-0.4) node[anchor=center, scale = 0.7] {0};
        \draw (1,-1.4) node[anchor=center, scale = 0.7] {0};
        \draw (2,-1.4) node[anchor=center, scale = 0.7] {0};
        \draw (3,-1.4) node[anchor=center, scale = 0.7] {0};
        \draw (4,-1.4) node[anchor=center, scale = 0.7] {1};
        \draw (5,-1.4) node[anchor=center, scale = 0.7] {0};
        \draw (6,-1.4) node[anchor=center, scale = 0.7] {0};
        \draw (1.5,-2.4) node[anchor=center, scale = 0.7] {1};
        \draw (2.5,-2.4) node[anchor=center, scale = 0.7] {0};
        \draw (3.5,-2.4) node[anchor=center, scale = 0.7] {1};
        \draw (4.5,-2.4) node[anchor=center, scale = 0.7] {0};
        \draw (5.5,-2.4) node[anchor=center, scale = 0.7] {0};
        \draw (2,-3.4) node[anchor=center, scale = 0.7] {0};
        \draw (3,-3.4) node[anchor=center, scale = 0.7] {1};
        \draw (4,-3.4) node[anchor=center, scale = 0.7] {-1};
        \draw (5,-3.4) node[anchor=center, scale = 0.7] {1};
        \draw (2.5,-4.4) node[anchor=center, scale = 0.7] {0};
        \draw (3.5,-4.4) node[anchor=center, scale = 0.7] {0};
        \draw (4.5,-4.4) node[anchor=center, scale = 0.7] {0};
        \draw (3,-5.4) node[anchor=center, scale = 0.7] {0};
        \draw (4,-5.4) node[anchor=center, scale = 0.7] {0};
        \draw (3.5,-6.4) node[anchor=center, scale = 0.7] {0};

        \draw[blue, dashed] (2,0) -- (4,-4);
        \draw[blue, dashed] (6,0) -- (4,-4);
        \draw[blue, dashed] (2,0) -- (6,0);
        \draw[fill=cyan, nearly transparent]  (2,0) -- (6,0) -- (4,-4) -- cycle;        
    
        \draw[blue, dashed] (1.5,-3) -- (4.5,-3);
        \draw[blue, dashed] (3,-6) -- (4.5,-3);
        \draw[blue, dashed] (1.5,-3) -- (3,-6);
        \draw[fill=cyan, nearly transparent]  (3,-6) -- (1.5,-3) -- (4.5,-3) -- cycle;        
    
        \draw[blue, dashed] (3.5,-3) -- (5.5,-3);
        \draw[blue, dashed] (4.5,-5) -- (5.5,-3);
        \draw[blue, dashed] (3.5,-3) -- (4.5,-5);
        \draw[fill=cyan, nearly transparent]  (3.5,-3) -- (5.5,-3) -- (4.5,-5) -- cycle;        
        
        \end{tikzpicture}
        \caption{Corollary 1}
        \label{Figure 9}
    \end{figure}


\begin{thebibliography}{9}
	
		\bibitem{robbins} 
		D. Robbins and H. Rumsey. “Determinants and Alternating-Sign Matrices,”
		\textit{Advances in Mathematics}. \textbf{62} (1986), 169-184.
		
		\bibitem{propp} 
		J. Propp. “The Many Faces of Alternating-Sign Matrices,”
		\textit{Discrete Mathematics and Theoretical Computer Science Proceedings}.
		\textbf{AA (DM-CCG)}
		(2001), 43-58.
		
		\bibitem{carroll} 
		G. Carroll and D. Speyer. “The Cube Recurrence,”
		\href{https://arxiv.org/abs/math/0403417}{https://arxiv.org/abs/math/0403417}
		
	\end{thebibliography}
\end{document}